\documentclass[11pt]{amsart}

\usepackage[top=1in, bottom=1in, left=1in, right=1in]{geometry}

\usepackage{graphicx}
\usepackage{amsmath}
\usepackage{amsfonts}
\usepackage{amssymb}
\usepackage{mathrsfs}
\usepackage[bb=boondox]{mathalfa}
\usepackage{hyperref}
\usepackage{enumerate}

\newtheorem{theorem}{Theorem}[section]
\newtheorem{proposition}[theorem]{Proposition}
\newtheorem{lemma}[theorem]{Lemma}

\theoremstyle{remark}
\newtheorem{remark}[theorem]{Remark}

\theoremstyle{definition}
\newtheorem{definition}{Definition}[section]

\newcommand{\abs}[1]{\left\vert #1 \right\vert}

\title[Self-Interacting Random Walks]{Scaling Limit of Asymptotically-free Self-interacting Random Walks to Brownian Motion Perturbed at Extrema}

\author{Xiaoyu Liu}
\address{Xiaoyu Liu, Purdue University, United States of America.}
\email{liu2999@purdue.edu}

\author{Zhe Wang}
\address{Zhe Wang, \'{E}cole Polytechnique F\'{e}d\'{e}rale de Lausanne, Switzerland.}
\email{zhe.wang@epfl.ch}

\thanks{Xiaoyu Liu gratefully acknowledges the support from the Purdue Mathematics Department through the Joel Spira Research Fellowship (2023).}
\thanks{Zhe Wang gratefully acknowledges the support from the Swiss National Science Foundation, grant 200021L--169691.}
\thanks{Both authors thank Prof. Jonathon Peterson and Prof. Thomas Mountford for their suggestion of the problem and for enriching discussions.}


\begin{document}

\begin{abstract}
	We consider a family of one-dimensional self interacting walks whose dynamics characterized by a monotone weight function $w$ on $\mathbb{N}\cup \{0\}$. The weight function takes the form $w(n) = (1 + 2^p Bn^{-p} + O(n^{-1-\kappa}))^{-1}$, for some $B \in \mathbb{R} $, $\kappa>0$ and $p\in (0,1]$.
	Our main model parameter is $p$, and for $p\in (0,1/2]$ we show the convergence of the SIRW to Brownian motion perturbed at extrema under the diffusive scaling. This completes the functional limit theorem in [8] for the asymptotically free case and extends the result to the full parameter range $(0,1]$.
	Our method depends on the generalized Ray-Knight theorems (\cite{T96}, \cite{KMP23}) for the rescaled local times of this walk.
	The directed edge local times, described by the branching-like processes, are used to analyze the total drift experienced by the walker. 
\end{abstract}

\maketitle

\section{Introduction and Main Result}

We consider a discrete time nearest neighbor self-interacting random walk (SIRW) $X = (X_k)_{k\geq 0}$ on $\mathbb{Z}$ as in \cite{T96,KMP23}. More precisely, let $\Omega = \left\{\omega = (\omega_i)_{i \geq 0}: \omega_0 = 0, |\omega_i - \omega_{i - 1}| = 1 \right\}$ be the set of all nearest neighbor paths starting at 0, and let $\mathcal{F}_i, i \ge 0$ be the $\sigma$-algebra generated by all sets of the form $\{\omega_j = x_j, j = 0, 1, \ldots, i\} $, and $\mathcal{F} = \sigma\left( \cup_{i = 0}^\infty   \mathcal{F}_i\right)$. For each $\omega$ we define the undirected edge local times

\[ 
l_x^k(\omega) = \sum_{j=0}^{k-1} \mathbb{1}_{ \left\{  \left\{\omega_j, \omega_{j+1}\right\} =  \left\{x,x-1\right\} \right\} }, \qquad
r_x^k(\omega) = \sum_{j=0}^{k-1} \mathbb{1}_{ \left\{  \left\{\omega_j, \omega_{j+1}\right\} =  \left\{x,x+1\right\} \right\} }   
,\]
where $l_x^k$ and $r_x^k$ are the local times by time $k$ for the undirected edges $\left\{x,x-1\right\}$ and $\left\{x,x+1\right\}$.
The dynamic of the walk $X$ under $(\mathbb{P}, \Omega)$ starting from $X_0 = 0$ is assumed to be given by the following conditional probabilities:

\begin{align}
	\mathbb{P}\left( X_{k+1} =  X_k+1 \middle| \mathcal{F}_k   \right) 
	&=1- \mathbb{P}\left( X_{k+1} =  X_k-1 \middle|  \mathcal{F}_k  \right)  
	\notag
	\\
	&=  \frac{  w(r_{X_k}^k)}{ w(l_{X_k}^k)  + w(r_{X_k}^k)   }
	, \label{eq: dynamic}
\end{align}
where $
w: \mathbb{N} \to  (0, \infty )
$ 
is a weight function.
We note that the walk defined as such is not Markov.
In this model, we require $w(.)$ to be monotone and converging to one, with polynomial asymptotic behavior:
\begin{equation}\label{eq: asymptotics of w}
	\frac{1}{w(n)} = 1 + \frac{2^p B}{n^p} + O\left(\frac{1}{n^{1+\mathcal{\kappa}}}\right) \quad \mbox{as $n\to \infty$} 	
\end{equation} 
for some $p \in (0,1]$, $\kappa>0$ and $B\in \mathbb{R}$. Under the monotonicity assumption of $w(.)$, $X_k$ is self-attracting if $w(.)$ is increasing and self-repelling if $w(. )$ is decreasing.


This type of model for SIRW was first introduced in \cite{T96} with different classes of asymptotic behaviors. ``Asymptotically free'' SIRWs, in B. T\'oth's original terminology, refers to those walks with weight functions of the type $\frac{1}{w(n)} = 1 + B n^{-1} + O(n^{-2})$. 
The current assumption \eqref{eq: asymptotics of w} on $w(n)$ is introduced in \cite{KMP23}, where the new parameter $p$ allows more general $w(.)$ and their results rely on the range of $p$. The interesting range of $p$ is $p \in (0,1]$, and $p > 1$ is absorbed into the error term in \eqref{eq: asymptotics of w}.
If instead of letting $w(n)$ converge to $1$, we let $w(n) \sim n^{-\alpha}$ for some $\alpha >0$, we obtain the ``polynomially self-repelling'' case also studied in \cite{T96}. 
We may view these two regimes as results of continuous variation of the parameter $\alpha$ in the general weight function $1/w(n) = n^{- \alpha}(1 + 2^p B n^{-p} + O(n^{-1-\kappa}))$.
There have been some works in characterizing the scaling limit of SIRWs in different regimes \cite{T96, KMP23}.

A key element in T\'oth's analysis of SIRW is the connection between SIRW and its local time, which is in spirit, similar to how the Ray-Knight theorem connects the local times of a Brownian motion with certain squared Bessel processes. 
This generalized Ray-Knight approach establishes convergence of local time processes of SIRWs, under scaling, to squared Bessel processes of appropriate dimensions. These Ray-Knight type results provide information about the scaling limits of the original SIRWs.
In \cite{T96}, Theorem~2A and 2B, T\'oth also proved local limit theorems for the position of a random walker in different SIRWs at independent geometric times with means of linear growth. 
These local limit theorems identified the diffusive limit of the endpoint of the SIRW as a Brownian motion perturbed at extrema (BMPE) (see Definition~\ref{defn:BMPE} below), provided the limit exists. 
The general question of whether there exist universal conditions that allow one to pass from Ray-Knight type results to a functional limit theorem was since open.
Based on these results, in \cite{KMP23}, Kosygina, Mountford and Peterson showed the process-level convergence of rescaled asymptotically free SIRWs to BMPEs under the assumption that $p > \frac{1}{2}$. Interestingly, they also constructed a counterexample to the functional convergence of rescaled SIRW in the polynomially self-repelling case, proving that there is no general thereom to pass directly from scaling limit of local time processes to functional convergence of the walk.
This non-convergence result also rejects the natural candidate, Brownian motion perturbed at extrema, as possible weak limits of the rescaled SIRWs, which is in contrast to those asymptotically free SIRWs with $p>1/2$. It's natural to ask whether the non-convergence result would occur for asymptotically free SIRWs when $p$ is close to $0$ and $w(.)$ is decreasing.
In the current work, we give an answer to this question by showing that the functional convergence of rescaled asymptotically free SIRWs to BMPEs holds for the full range of $p \in (0,1]$.




\subsection{Main Result}
We recall a definition for a Brownian motion perturbed at extrema, which is the scaling limit of $(X_k)_{k\geq 0}$.
\begin{definition}
	\label{defn:BMPE}
	Let $\theta^+, \theta^- \in (- \infty , 1)$. A Brownian motion perturbed at
	extrema (BMPE) $W^{\theta^+, \theta^-} = \left(W^{\theta^+, \theta^-}_t\right)_{t\geq 0}$ with parameter $(\theta^+, \theta^-)$ is the pathwise unique solution of the equation
	\[
	W_t = B_t \,+\, \theta^+ \sup_{s\leq t} W_s  \,+\, \theta^- \inf_{s\leq t} W_s \,,   \qquad t \ge 0, \quad W_0 = 0,
	\]
	where $B_t$ is a standard Brownian motion.
\end{definition}
It was shown in \cite{PW97, CD99} that if $\theta^+, \theta^- < 1$ then almost surely the functional equation above has a pathwise unique solution that is continuous, adapted to the Brownian filtration, and has Brownian scaling. 
Furthermore, the process triple 
\[
	\left(\inf_{s < t} W^{\theta^{+}, \theta^{-}}_s\!, \quad
	W^{\theta^{+}, \theta^{-}}_t\!, \quad
	\sup _{s<t} W^{\theta^{+}, \theta^{-}}_s\right)_{t \geq 0}
\]
is a strong Markov process (see Theorem~1 in \cite{PW97}).

Following the notation in \cite{T96}, we define
\[
U_1(n):=\sum_{j=0}^{n-1}(w(2 j))^{-1} \quad \text{and} \quad
V_1(n):=\sum_{j=0}^{n-1}(w(2 j+1))^{-1}
\]
and set
\begin{equation}
	\label{eq: gamma}
	\gamma:= \lim_{m\to \infty}\left( V_1(m) - U_1(m) \right) =\lim_{m\to \infty} \left( \sum_{j=0}^{m-1} \frac{1}{ w(2j+1)}-  \sum_{j=0}^{m-1}  \frac{1}{w(2j)} \right) 
	.\end{equation}
Motivation for these definitions will become apparent in the proof of Lemma~\ref{lm: convergence of mean of discrepancies}.
It is known that for monotone $w(.)$, $\gamma$
is well-defined and ${\gamma<1}$; in particular, $\gamma<1$ implies that the model is recurrent (\cite{KMP23}, p4 and \cite{T96}, p1340). 
 

	The following functional limit theorem is shown in \cite{KMP23}: 
	\begin{theorem}[\cite{KMP23}, Theorem~1.1]
		Let $w(.)$ be monotone and satisfy \eqref{eq: asymptotics of w} for $p\in (\frac{1}{2},1]$, and $\kappa >0 $. Consider the SIRW $(X_k)_{k\geq 0}$ defined in \eqref{eq: dynamic} with $X_0 =0$. We have the following process-level convergence
		\[
		\left(  \frac{X_{\lfloor nt \rfloor }}{\sqrt{n}}  \right)_{t\geq 0} \Longrightarrow \left( W^{\gamma,\gamma}_{t}\right)_{t\geq 0},
		\] 
		as $n$ goes to infinity, in the standard Skorohod topology on $D([0,\infty) ).$
	\end{theorem}
\begin{samepage}
	Our main result is a functional limit theorem in the case when $p\in (0,\frac{1}{2}]$.
	\begin{theorem}\label{th: main}
		Let $w(.)$ be monotone and satisfy \eqref{eq: asymptotics of w} for $p\in (0,\frac{1}{2}]$, and $\kappa >0 $. Consider the SIRW $(X_k)_{k\geq 0}$ defined in \eqref{eq: dynamic} with $X_0 =0$. We have the following process-level convergence
		\[
		\left(  \frac{X_{\lfloor nt \rfloor }}{\sqrt{n}}  \right)_{t\geq 0} \Longrightarrow \left( W^{\gamma,\gamma}_{t}\right)_{t\geq 0},
		\]
		as $n$ goes to infinity, in the standard Skorohod topology on $D([0,\infty) ).$
	\end{theorem}
\end{samepage}

This theorem states that the amount of perturbation is directly proportional to the signed range of the limiting process, i.e. $W_t = B_t + \gamma \left( \sup_{s \le t} W_t + \inf _{s \le t} W_s \right) $. 
One way to establish this is to approximate the drift encountered by the walk by $\gamma$ times the signed range of the walk, with an error of order $o\left(\sqrt{n} \right)$ (see Lemma~\ref{lm: control of acc drift} for the precise statement). 
The approximation is intuitive when the drift is experienced only at the extrema, for example, when $w(0) = (1 - \gamma)^{-1}$ and $w(n) = 1$ for $n > 0$, see \cite{Dav99} for a dicussion of this case. 
As our proof will show, this approximation remains true in general for $w(.)$ defined in \eqref{eq: asymptotics of w}, for all $p \in (0,1]$, because the average of local drifts converges to $\gamma$ fast enough and the walker makes enough visits to each site in its range (this will be made precise in Lemma~\ref{lm: convergence of mean of discrepancies} and Lemma~\ref{lm: number of rarely visit sites} below).
In particular, our proof of Theorem 1.3 can also recover the case when $p\in (1/2,1]$.

The functional convergence of rescaled SIRWs to BMPEs has also been shown for other non-Markov one-dimensional random walks,
such as once-reinforced random walks, excited random walks, and rotor walks with defects, see \cite{Dav96,Dav99,DK12,KP16,KMP22,HLSH18} and references therein. 
In particular, the excited random walks (ERWs) have interactions through local times of sites instead of local times of undirected edges, and their studies have been generalized to the context of random walks in random environment by allowing weight functions to be random and site-dependent, 
see \cite{KZ13, KMP22}.
Our result should be compared to the recurrent ERWs in the non-boundary case, see \cite{KP16,KMP23}, because the properties of walks are similar and the approaches are comparable. 
Those ERWs in the recurrent non-boundary case have diffusive scaling for their position and the local time processes; in particular, BMPEs and squared Bessel processes are only possible functional limits (as observed in \cite{T96}). In contrast, a different scaling limit under a non-diffusive scaling is expected when the walk is transient or in the boundary recurrent case. We also remark that $\gamma<1$ in our model implies that we are in the non-boundary recurrent case.

The analysis of local time processes of random walks via \textit{branching-like processes} (BLPs), taking advantage of the tree structure of the walk's excursions in one dimension, is core to the analysis of both SIRW (considered in this work as well as \cite{T96, KMP23}) and ERW (considered in \cite{DK12, KP16, KMP22}). 
One special property of SIRW is that the local behavior of the walk at different sites can be generated from independent \textit{generalized P\'olya's urn processes} assigned to each site.
Analysis of these two auxiliary processes is central in estimating the total drift attained by the SIRW in both \cite{KMP23} and this work.
In the case $p \in (\frac{1}{2}, 1]$, the total drift accumulated at any site $x$ turns out to be absolutely summable, which allows one to estimate the total drift easily (\cite{KMP23}, Lemmas~2.3-2.4).
The main technical difficulty in proving Theorem~\ref{th: main} where $p \le \frac{1}{2}$, is that we no longer have absolute summability of local drifts.
It turns out that we can estimate the accumulated drift without using absolute summability of the local drifts, if we stop the walk at particular excursion times and express the total attained drift as a spatial martingale adapted to the natural filtration of a branching-like process. After that, we control this martingale by considering a typical event on which the martingale has bounded increments.

\subsection{Organization of the Paper}
In section \ref{sec: proof of main}, we outline the proof of Theorem \ref{th: main} and reduce the proof into several technical lemmas. The proof of these lemmas are postponed to section \ref{sec: approximations} because they involve analysis of auxiliary processes. In section \ref{sec: generalized Polya Urn, BLP}, we describe the generalized P\'{o}lya urn processes and branching-like processes, discuss their connections to the SIRW $(X_k)_{k\geq 0}$, and recall some preliminary results from previous works which will be used in the proof of Theorem~\ref{th: main}. In section \ref{sec: approximations}, we prove the technical lemmas stated in section \ref{sec: proof of main} and thus finish the proof of Theorem~\ref{th: main}. 

\subsection{Notation}


	Here we collect some notation related to the SIRW $(X_k)_{k \ge 0}$.
	We denote by $\mathbb{P}$ the probability measure on $D([0,\infty))$ induced by $(X_k)_{k\geq 0}$, and by $\mathbb{E}$ the expectation with respect to the probability measure $\mathbb{P}$. Here we have identified the integer-valued process $(X_k)_{k \ge 0} \in (\mathbb{P}, \Omega)$ as defined in the introduction, with a real-valued process $(X_{\lfloor t \rfloor})_{t \ge 0}$ living in the Skorohod topology $D([0,\infty))$. 
\begin{enumerate}
	\item Running maximum and minimum: For the sequence $(X_k)_{k \ge 0}$, the running maximum at time $k$ is defined as $S_k = \sup\{X_i : i \le k\}$. Similarly, the running minimum at time $k$ is given by $I_k = \inf \left\{X_i: i \le k\right\}$.
	\item
	Local time at a site: For any site $x \in \mathbb{Z}$ and time $j \in \mathbb{N}_0$, the local time is defined as $L(x,n):= \sum_{i=0}^n \mathbb{1}_{\{X_i=x\}}$. 
	This definition, together with $\lambda_{x, m}$ defined below, are minor variations of those in \cite{KMP23} and \cite{KP16}, in that we include the last step $X_n$.
	\item
	Time of return to a site: For any non-negative integer $m$, the time of the $m$-th return (or the $(m+1)$-th visit) to a site $x$ is denoted by $\lambda_{x,m} = \inf\{t \geq 0: L(x,t) = m+1\}$. 
	We note that for $x\neq 0$, $\lambda_{x,0} > 0$. Also, every $\lambda_{x,m}<\infty$ because the walk is recurrent, see \cite{T96}, p1325.
	\item
	Local time of directed bonds: For a directed bond $(x,x+1)$ and time $j \in \mathbb{N}_0$, the local time along the directed bond $(x, x+1)$ at time $j$ is defined as $\mathcal{E}^j_{x,+} = \sum_{i=0}^{j-1} \mathbb{1}_{\{X_i=x, X_{i+1} =x+1\}}$. 
	Similarly, the local time along $(x,x-1)$ at time $j$ is $\mathcal{E}^j_{x,-} = \sum_{i=0}^{j-1} \mathbb{1}_{\{X_i=x, X_{i+1} =x-1\}}$.
	We denote by $\mathcal{E}^{(x,m)}_{y, -}$ the local time $\mathcal{E}^{\lambda_{x,m}}_{y,-}$, and by by $\mathcal{E}^{(x,m)}_{y, +}$ the local time $\mathcal{E}^{\lambda_{x,m}}_{y,+}$.
	\item
	Stopping times: For any stochastic process $(Y_k)_{k \in \mathbb{N}_0}$ we write $\tau^Y_{i} = \inf \{k \in  \mathbb{N}_0: Y_k \ge  i\}$.
\end{enumerate}
\section{Functional Limit Theorem for AF-SIRW: Proof of Theorem \ref{th: main}}
\label{sec: proof of main}

The proof of Theorem \ref{th: main} follows a classical strategy in obtaining functional limit theorems,
which has been applied to a number of excited random walk models, see \cite{KP16} and \cite{KMP23}.
We start by decomposing the random walk. Let
\begin{equation}
	\label{eq:decomposition}
	X_k = M_k+ \Gamma_k 
	,\end{equation} 
where
\[ 
\Gamma_0 = 0, \quad \Gamma_n = \sum_{i=0}^{n-1} \mathbb{E}\left[ X_{i+1}-X_i | \mathcal{F}_i
\right].
\]
The above decomposition gives a martingale $M_k$ with respect to $\mathcal{F}_k$.
Then we divide the proof into four main steps. 

\vspace{1em}

\textbf{Step 1: Control of the martingale term.}
By the martingale functional limit theorem (\cite{B99}, Theorem~18.2 ), the rescaled process $\left( \frac{M_{\left\lfloor n t \right\rfloor}}{\sqrt{n}} \right) _{t \ge 0}$ is tight in $D\left( [0,\infty ) \right) $ and converges to the standard Brownian motion, if we have the following control in the martingale's quadratic variation:
\begin{equation}\label{eq: QV term}
	\lim_{N\to \infty}\frac{1}{N} \sum_{k=0}^{N-1}\mathbb{E}\left[ (M_{k+1}- M_{k})^2 |\mathcal{F}_k \right] =1 \qquad  \mbox{ in probability}.
\end{equation}

Since $\abs{X_{k+1}-X_k}=1$, \eqref{eq: QV term} is implied by the following estimate:
\begin{lemma} \label{lm: control of martingale} 
	Let $p\in (0,\frac{1}{2}]$. Then, for any $\varepsilon >0$
	\begin{equation}
		\lim_{N \to \infty }\mathbb{P}\left(\frac{1}{N} \sum_{k = 0}^{N-1} \mathbb{E}\left[ X_{k+1} - X_k | \mathcal{F}_k \right]^2 > \varepsilon \right) =0. 
	\end{equation}
\end{lemma}
The proof of Lemma~\ref{lm: control of martingale} is given in section \ref{sec: approximations}.

\textbf{Step 2: Control of the accumulated drift.} This is the major technical step of our article. We want to approximate the accumulated drift $\Gamma_k$ by a linear combination of the running maximum and minimum of the walk.
\begin{lemma}\label{lm: control of acc drift}
	Let $p\in (0,\frac{1}{2}]$. Then, for any $t>0$ and $\varepsilon >0$
	\begin{equation}
		\lim_{n \to \infty }\mathbb{P}\left(\sup_{k\leq \lfloor nt\rfloor} \abs{\Gamma_k - \gamma \left(S_k + I_k \right)   } > \varepsilon \sqrt{n}  \right) =0. 
	\end{equation}
\end{lemma}

To prove Lemma~\ref{lm: control of acc drift}, we will consider the description of the local time process of $(X_k)_{k \ge 0}$ at excursion times $\lambda_{x,m}$ in terms of spatially Markov processes, namely the branching-like processes (BLPs) $\left( \mathcal{E}^{(x,m)}_{z,+} \right)_{x \le z \leq y}$, to be discussed in section \ref{sec: generalized Polya Urn, BLP}. This motivates considering the \textit{local drift} attained at a single site $y$ at the stopping times $\lambda_{x, m} = k$, for various possible values of $(x,m)$:
\begin{equation}\label{eq: accumulated local drift}
	\Delta_y^{(x,m)}:= \sum_{i=0}^{\lambda_{x,m}-1} \mathbb{E}\left[X_{i+1}-X_i\vert \mathcal{F}_{i}\right] \mathbb{1}_{\left\{X_i=y\right\}}
\end{equation}
On the event that $\lambda_{x,m} = k$, the drift term $\Gamma_k$ is by definition a sum of local drifts:
\begin{equation}\label{eq: drift in terms of local drifts}
	\Gamma_k = \sum_{y\in \mathbb{Z}} \Delta_y^{(x,m)}
	.\end{equation}
We decompose the accumulated drift $\Gamma_k$ into three parts: $\Gamma_k = 	\Gamma_k^+ +	\Gamma_k^0 + \Gamma_k^-$, where 
\[
\Gamma_k^{+} = \sum_{y > x} \Delta_y^{(x,m)}\qquad 
\Gamma_k^{0} = \Delta_x^{(x,m)} \qquad
\Gamma_k^{-} = \sum_{y < x} \Delta_y^{(x,m)}
.\]
We will approximate $\Delta_y^{(x,m)}$ by $\gamma\cdot sgn(y)$, and thus we will approximate $\Gamma_k^+$ by $\gamma \cdot (S_k - \abs{X_k})$ and $\Gamma_k^-$ by $ \gamma \cdot (I_k + \abs{X_k} )$
.

For each $k\leq \lfloor nt\rfloor$, we can let $x = X_k$ and $ m +1=L(x,k)$ so that $\lambda_{x,m} = k$. From the decomposition \eqref{eq: drift in terms of local drifts} of $\Gamma_k$, we see that
Lemma \ref{lm: control of acc drift} will follow if we can prove
\begin{equation}\label{eq: control of acc drift + }
	\lim_{n \to \infty }\mathbb{P}\left(\sup_{k\leq\lfloor nt \rfloor} \abs{\Gamma^+_k - \gamma \cdot \left(S_k - \abs{X_k} \right)   } > \varepsilon \sqrt{n}  \right) =0, 
\end{equation}
and similar results for $\Gamma_k^0$ and $\Gamma_k^-$.
By reflection symmetry about site $0$ ($x \mapsto -x$), the result for $\Gamma_k^-$ is proved identically as that of $\Gamma_k^+$, and is thus omitted. 
$\Gamma_k^0$ will be dealt with when we prove the result for $\Gamma_k^+$.
Next, we note that \eqref{eq: control of acc drift + } is equivalent to
\begin{equation}\label{eq: control of acc drift ++}
	\lim_{n \to \infty }\mathbb{P}\left(\sup_{k\leq\lfloor nt \rfloor} \abs{\sum_{y> X_k} \left( \Delta_{y}^{(X_k,L(X_k,k) - 1)} - \gamma  \cdot sgn(y) \right)   }  > \varepsilon \sqrt{n}  \right) =0. 
\end{equation}

To show \eqref{eq: control of acc drift ++}, we further replace $\Delta_{y}^{(x,m)}$ by its conditional expectation with respect to the filtration $\left(\mathcal{G}_{y}^{(x,m)}\right)_{y\geq x}$ generated by directed edge local times, i.e. BLPs, $ \mathcal{G}_{y}^{(x,m)} = \sigma\left( \mathcal{E}^{(x,m)}_{z,+} : x \le z \leq y \right)$.
Define
\begin{equation}\label{eq: conditional mean}
	\rho_{y}^{(x,m)}= \mathbb{E}\left[\Delta_y^{(x,m)} | \mathcal{G}_{y-1}^{(x,m)}\right],
\end{equation}
then \eqref{eq: control of acc drift + } follows from the following two lemmas.
\begin{lemma}\label{lm: approximation of means of local drift}
	Let $p\in (0,\frac{1}{2}]$. Then, for any $t>0$ and any $\varepsilon >0$
	\begin{equation}
		\lim_{n \to \infty }\mathbb{P}\left(\sup_{k\leq\lfloor nt \rfloor} \abs{\sum_{y> X_k} \left( \rho_{y}^{(X_k,L(X_k,k)-1)} - \gamma  \cdot sgn(y) \right)   }  > \varepsilon \sqrt{n}  \right) =0. 
	\end{equation}
\end{lemma}

\begin{lemma}\label{lm: approx local drift by conditional means}
	Let $p\in (0,\frac{1}{2}]$. Then, for any $t>0$ and any $\varepsilon >0$
	\begin{equation}
		\lim_{n \to \infty }\mathbb{P}\left(\sup_{k\leq\lfloor nt \rfloor} \abs{\sum_{y> X_k} \left(\Delta_{y}^{(X_k,L(X_k,k)-1)}- \rho_{y}^{(X_k,L(X_k,k)-1)} \right)   }  > \varepsilon \sqrt{n}  \right) =0. 
	\end{equation}
\end{lemma}



To prove Lemmas \ref{lm: approximation of means of local drift} and \ref{lm: approx local drift by conditional means}, we will introduce the auxiliary processes associated to the SIRW $(X_k)_{k\geq 0}$: the generalized P\'{o}lya urn process, and the branching-like processes. Both processes are constructed from path of $X_k$ up to certain stopping times, and they are used in approximating $\Delta_{y}^{(x,m)}$.
It is natural to view these two processes and conditional expectations with respect to them as ``mesoscopic quantities'', in between $\Delta_y$ and $\gamma\cdot sgn(y)$.
A major challenge in our work is to understand how fluctuations around the conditional expectations vanish in terms of Lemmas~\ref{lm: approximation of means of local drift} and \ref{lm: approx local drift by conditional means}.
To address this challenge, we work on certain good events to be defined in section 4.1.
On one hand, the good events $A_{n, K}^c \cap B_{n, M}^c$, dealing with upper and lower control of (directed) local times, used in the proof of Lemma~\ref{lm: approximation of means of local drift}, are essentially known from \cite{KMP23}.
On the other hand, the good events $G_{n, K, t}$ and $G_{n, K}^{(x, m)}(y)$ for Lemma~\ref{lm: approx local drift by conditional means} require a new construction, and a finer analysis of edge local times and underlying generalized P\'olya’s Urn processes, which is the major novelty of this article. 
Lemmas~\ref{lm: approximation of means of local drift} and \ref{lm: approx local drift by conditional means} together with the constructions of good events are proved in section \ref{sec: approximations}.

\vspace{1em}

\textbf{Step 3: Tightness.} The tightness of $S_k$ and $I_k$ under diffusive scaling for general $p \in (0,1]$ is already established in \cite{KMP23}:
\begin{proposition}
	(\cite{KMP23}, Proposition 2.1)\\
	\label{prop: tightness}
	The scaled extremal processes $\left\{\frac{S_{\left\lfloor n t \right\rfloor}}{\sqrt{n}}\right\}_{n \geq 0}$ and $\left\{\frac{I_{\lfloor n t \rfloor}}{\sqrt{n}}\right\}_{n \geq 0}$ are tight in the standard Skorohod topology on $D([0, \infty))$.
\end{proposition}
This result combined with Lemma~\ref{lm: control of acc drift} gives tightness of $\left(\frac{\Gamma_{\left\lfloor nt  \right\rfloor}}{\sqrt{n} }\right)_{t \ge 0}$, and hence of $\left(\frac{X_{\left\lfloor nt  \right\rfloor}}{\sqrt{n} }\right)_{t \ge 0}$. 
Now we are ready to show
\vspace{1em}

\textbf{Step 4: Convergence to BMPE.} 
From Proposition~\ref{prop: tightness} and Lemmas \ref{lm: control of martingale} and \ref{lm: control of acc drift} we can conclude that the process triple $\frac{1}{\sqrt{n}}\left(X_{\lfloor n t\rfloor}, M_{\lfloor n t\rfloor}, \Gamma_{\lfloor n t\rfloor}\right)_{t \geq 0}$ is tight in the space $D([0, \infty))^3$ and that any subsequential limit $\left(Y_1(t), Y_2(t), Y_3(t)\right)_{t \geq 0}$ is a continuous process such that $Y_2$ is a standard Brownian motion, $Y_3(t)=$ $\gamma\left(\sup _{s \leq t} Y_1(s)+\inf _{s \leq t} Y_1(s)\right)$ for all $t \geq 0, \mathbb{P}$-a.s., and $Y_1(t)=Y_2(t)+Y_3(t)$. By uniqueness of functional solution for BMPE, $\frac{1}{\sqrt{n} } X_{\left\lfloor nt  \right\rfloor}$ converges to a $(\gamma, \gamma)$-BMPE.

\section{Generalized P\'{o}lya Urn, Branching-Like Processes}\label{sec: generalized Polya Urn, BLP}

In this section, we describe two auxiliary processes: the generalized P\'{o}lya urn processes and the branching-like processes (BLPs). As explained earlier in section \ref{sec: proof of main}, these two auxiliary processes are essential in estimating $\Gamma_k^+= \sum_{y\geq X_k} \Delta_{y}^{(x,m)}$ on the event $\{\lambda_{x,m} = k\}$. We will show in subsection \ref{subsec: measurability} that $\Delta^{(x,m)}_{y}$ is itself not measurale with respect to the natural filtrations of BLPs, $\mathcal{G}$, yet its conditional expectation $\rho^{(x,m)}_{y}$ only depends on the generalized P\'{o}lya urn processes associated to site $y$. Therefore, to approximate $\Gamma_k$ and $\Delta_{y}^{(x,m)}$ we need to consider a larger spatial filtration $\mathcal{H}$, in which $\Delta^{(x,m)}_{y}$ becomes a Markov process with some convenient properties. In the last subsection, we recall some properties of the two processes needed in the proofs of Lemma \ref{lm: control of martingale}, \ref{lm: approximation of means of local drift}, and \ref{lm: approx local drift by conditional means}. Most of the these properties are well-known from works such as \cite{T96}, \cite{KP16} and \cite{KMP23}. 


\subsection{Generalized P\'{o}lya Urn}
Given a (recurrent) SIRW $(X_k)_{k\geq 0}$, and a fixed site $y\in \mathbb{Z}$, we can obtain a Markov process by considering only the up-crossings and down-crossings of $X_k$ from site $y$. More precisely, we first let $\lambda_{y,i}$ be the stopping time when $X_k$ visit site $y$ for the $\left( i+1 \right) $-th time:
\[
\lambda_{y,0} :=\inf\left\{ t\geq 0: X_t = y \right\} , \quad \lambda_{y,i+1} := \inf\left\{ t> \lambda_{y, i}: X_t = y \right\},
\] 
and then define the \textit{generalized P\'olya urn process} at site $y$ as 
\begin{equation} \label{eq: RW to GPU}
	\left(\mathscr{B}^{(y)}_{i},\mathscr{R}^{(y)}_{i} \right)_{i\ge 0}
	:=\left(\mathcal{E}^{\lambda_{y,i}}_{y,-}, \mathcal{E}^{\lambda_{y,i}}_{y,+}\right)_{i\geq 0} 
	=  \left(\mathcal{E}^{(y,i)}_{y,-}, \mathcal{E}^{(y,i)}_{y,+}\right)_{i\geq 0}.
\end{equation}
This is a Markov process with an initial value $(0,0)$. 
If we view this process as generated by repeatedly drawing blue and red balls, then
$\left(\mathscr{B}_{i}^{(y)},\mathscr{R}_{i}^{(y)} \right)$ is the state of a generalized P\'olya urn, after the $i$-th draw made from the urn, with $\mathscr{B}_{i}^{(y)}$ blue balls and $\mathscr{R}_{i}^{(y)}$ red balls. 
This explains our choice of defining $\lambda_{y, 0}$ as the first time the walk reaches $y$: it is right before the first draw is made from the urn.

Due to the initial value of the underlying random walk $X_0=0$, the local times of undirected edges $\left\{y,y-1\right\}$ and $\left\{y,y+1\right\}$ at time $\lambda_{y,0}$ are 
(recall definition from beginning of section~1):
\begin{equation}\label{eq: initial condition}
	\left(l_y^{\lambda_{y,0}},  r_y^{\lambda_{y,0}}\right) =  \begin{cases}	
		(1, 0) &,  \text{ if }  y>0 \\
		(0, 1) &,  \text{ if }  y<0 \\  
		(0, 0) &,  \text{ if }  y=0 \\
	\end{cases} 
	.\end{equation}	
Therefore, the transition probabilities of $\left(\mathscr{B}_{i}^{(y)},\mathscr{R}_{i}^{(y)} \right)$ defined in \eqref{eq: RW to GPU} are:
\begin{align*}\label{eq: transition prob for GPU}
	\mathbb{P} \left(\left(\mathscr{B}^{(y)}_{k+1},\mathscr{R}^{(y)}_{k+1} \right)=  (i+1,j) \vert \left(\mathscr{B}^{(y)}_{k},\mathscr{R}^{(y)}_{k}\right) =(i,j)  \right) &= \frac{b_y(i)}{b_y(i)+r_y(j)}, \mbox{ and}  \\
	\mathbb{P} \left( \left(\mathscr{B}^{(y)}_{k+1},\mathscr{R}^{(y)}_{k+1}\right)=  (i,j+1) \vert \left(\mathscr{B}^{(y)}_{k},\mathscr{R}^{(y)}_{k}\right) =(i,j)  \right) &= \frac{r_y(j)}{b_y(i)+r_y(j)},
\end{align*} 
where the weights $(b_y(k),r_y(k))_{k\geq 0}$ depend on $w(.)$ and $y$ as follows:
\begin{equation}\label{eq: generalized weights}
	(b_y(k), r_y(k)) = \begin{cases}
		(w(2k+1), w(2k)) &,  \text{ if }  y>0 \\
		(w(2k), w(2k+1)) &,  \text{ if }  y<0 \\  
		(w(2k), w(2k)) &,  \text{ if }  y=0 \\ 
	\end{cases}.
\end{equation}

For a generalized P\'{o}lya urn process $\left(\mathscr{B}^{(y)}_k,\mathscr{R}^{(y)}_k \right)_{k\geq 0}$ associated to the site $y$, we define the \textit{signed difference process} $\left\{\mathscr{D}^{(y)}_{k}\right\}_{k \ge 0} $ to be
\begin{equation}\label{eq:signed difference}
	\mathscr{D}^{(y)}_k  =\mathscr{R}^{(y)}_k -\mathscr{B}^{(y)}_k.  
\end{equation}
Let $\left(\mathcal{F}^{\mathscr{B},\mathscr{R}}_{k, y}\right)_{k \ge 0}$ be the natural filtration associated to the urn process $\left(\mathscr{B}^{(y)}_k,\mathscr{R}^{(y)}_k \right)_{k\geq 0}$:  
\[
\mathcal{F}^{\mathscr{B},\mathscr{R}}_{k, y} = \sigma\left( \left(\mathscr{B}_j^{(y)},\mathscr{R}_j^{(y)} \right): j\leq k \right).
\]  
For simplicity of notation, we may write $\tau_{k,y}^{\mathscr{B}}$ in place of $\tau_k^{\mathscr{B}^{(y)}}$,  $\tau_{k,y}^{\mathscr{R}}$ in place of $\tau_k^{\mathscr{R}^{(y)}}$, and $\mathcal{F}^{\mathscr{B},\mathscr{R}}_{k, y}$ in place of $\mathcal{F}^{\mathscr{B}^{(y)},\mathscr{R}^{(y)}}_{k}$. We may also drop the $y$ when there is no ambiguity.
\begin{remark}
	\label{rm:symmetry}
	By symmetry considerations (recall $X_0 = 0$), for all $y \in \mathbb{Z}$, the generalized P\'{o}lya urn processes at sites $y$ and $-y$ are symmetric in the sense that
	\[\left(\mathscr{B}^{(y)}_{k},\mathscr{R}^{(y)}_{k} \right)_{k\ge 0}
	\overset{d}{=} 
	\left(\mathscr{R}^{(-y)}_{k},\mathscr{B}^{(-y)}_{k} \right)_{k\ge 0} 
	.\]

	Moreover, for all $x, m \in \mathbb{Z}$ such that $m\geq 0$, we have
	\[
	\left(\mathcal{E}^{(x,m)}_{x+k,+} \right)_{k\geq 0} \overset{d}{=} \left(\mathcal{E}^{(-x,m)}_{-x-k,-} \right)_{k\geq 0}.
	\]
\end{remark}
\begin{remark}
	\label{rk:UrnGeo}
	Since the weight function $w(.)$ is positive and bounded away from zero and infinity, the probability of the next ball drawn being blue is bounded below by a constant $q > 0$ and above by a constant $q' < 1$. Therefore, for each site $y$ and for all $i \ge 0$, the difference $\tau_{i+1}^{\mathscr{B}} - \tau_{i}^{\mathscr{B}}$ is stochastically bounded above by a geometric random variable with parameter $q$ and below by a geometric random variable with parameter $q'$.
	In particular, each color is drawn from the urn infinitely often almost surely, which is also a consequence of recurrence of $X_k$.
\end{remark}

\subsection{Branching-Like Processes}
Having understood the walk in terms of the urn processes at a fixed site, we now characterize how urn processes at different sites are related. Unlike the generalized P\'{o}lya urn processes $(\mathscr{B}_k,\mathscr{R}_k )_{k \ge 0}$ that are Markov processes in time, the branching-like processes, describing the local times of directed edges of $(X_k)_{k\geq 0}$ at a fixed stopping time, are Markov processes in space.
More precisely, for any integers $x,m$ with $m\geq 0$, the local times at the stopping time $\lambda_{x,m}$, 
\[
\left(\mathcal{E}^{(x,m)}_{x+k,+} \right)_{k\geq 0}, \quad \left(\mathcal{E}^{(x,m)}_{x-k,-} \right)_{k\geq 0}
\]
are two Markov processes on $\mathbb{N}\cup\left\{0\right\}$. Their transition probabilities are related to the generalized weights in \eqref{eq: generalized weights}, see \eqref{eq: transition prob on positive} and \eqref{eq: transition prob on negative} below. The derivation is known from several earlier works, such as \cite{T96, KP16}. We state some facts in the derivation of \eqref{eq: transition prob on positive} and \eqref{eq: transition prob on negative} , which are also used in the next subsection.

Due to Remark~\ref{rm:symmetry}, its natural to assume that $x\ge 0$ and $m \ge 0$. We define the \textit{branching-like processes} to be
\[
\tilde{\zeta} := \left(\mathcal{E}^{(x,m)}_{x+k,+} \right)_{k\geq 0}, \quad
\zeta := \left(\mathcal{E}^{(x,m)}_{x-k,-} \right)_{k\geq 0}
.\]
In particular, $\tilde{\zeta}$ is a homogeneous Markov chain, while $\zeta$ is inhomogeneous. We have the following properties for $\tilde{\zeta}$:
\begin{enumerate}
	\item The sequences $(\tau^{\mathscr{B}}_{k,y})_{k\geq 0} $ are independent in $y \geq x$, and each sequence $\left(\tau^{\mathscr{B}}_{k,y}\right)_{k\geq 0} $ is Markov in $k$. Then, the collection of stopping times
	\begin{equation}\label{eq: markov 1} 
		\left\{\tau^{\mathscr{B}}_{k,y}: y\geq x, k\geq 0 \right\} \mbox{are Markov in $(y,k)$}
	\end{equation}
	under the lexicographical order (which is a total order on any subset of $\mathbb{Z}^2$): 
	\begin{equation*}\label{eq: lexicographical order}
		(y,k) \preceq (y',k')  \mbox{ if and only if }
		k \leq k'   \mbox{ when $y' = y$, or } 
		y <y'. 
	\end{equation*} 
	
	\item  For any $y\geq x$, 
	$\mathcal{E}^{(x,m)}_{y+1,+}$ is a function of $\{\tau^{\mathscr{B}}_{k,y+1}: k\leq \mathcal{E}^{(x,m)}_{y+1,-}\}$,
	\begin{equation} \label{eq: recursive formula for upcrossings}
		\mathcal{E}_{y+1,+}^{(x,m)}	=  \sum_{k= 0 }^{\mathcal{E}_{y+1,-}^{(x,m)}-1}	\left(\tau^{\mathscr{B}}_{k+1,y+1}-\tau^{\mathscr{B}}_{k,y+1}-1 \right) = \tau^{\mathscr{B}}_{ L,y } - L = \mathscr{R}^{(y + 1)}_{\tau^{\mathscr{B}}_{ L,y+1 }},
	\end{equation}
	where $L = \mathcal{E}_{y+1,-}^{(x,m)}$.
	
	\item The directed edge local times at two consecutive sites satisfy:
	\begin{equation}\label{eq: source of inhomogeneity}
		\mathcal{E}_{y-1,+}^{(x,m)} = \mathcal{E}_{y,-}^{(x,m)} + \mathbb{1}_{ \left\{ 1\leq y \leq x \right\} }
		\quad \text{for all }y\in \mathbb{Z}.
	\end{equation}
	
	\item  From \eqref{eq: markov 1}, \eqref{eq: recursive formula for upcrossings} and \eqref{eq: source of inhomogeneity}, the transition probabilities for $\tilde{\zeta}$ are 
	\begin{equation}\label{eq: transition prob on positive}
		\mathbb{P}\left(\tilde{\zeta}_{k+1}=j \vert \tilde{\zeta}_k =i  \right) = 
		\mathbb{P}\left( \mathscr{R}_{\tau_{i,1}^{\mathscr{B}}} = j \right), \mbox{for any $i,j\geq 0$, } 
	\end{equation} 
	where the urn processes are located at $y = 1$ (or equivalently, any site $y>0$).
\end{enumerate}

For $\zeta= \left(\mathcal{E}^{(x,m)}_{x-k,-} \right)_{k\geq 0}$, we get similar statements using Remark~\ref{rm:symmetry}, by reversing the roles of $+$ and $-$, as well as $\mathscr{B}$ and $\mathscr{R}$. $\zeta$ is inhomogeneous because of \eqref{eq: source of inhomogeneity}, and we have the transition probabilities depending on the sign of $x-k$: for $i,j\geq 0$
\begin{equation}\label{eq: transition prob on negative}
	\mathbb{P}\left(\zeta_{k+1}=j \vert \zeta_k =i  \right) = 
	\begin{cases}
		\mathbb{P}\left( \mathscr{B}_{\tau_{i+1,1}^{\mathscr{R}}} = j \right) ,& \mbox{ if $0 \leq k <  x-1$ }
		\\
		\mathbb{P}\left( \mathscr{B}_{\tau_{i+1,0}^{\mathscr{R}}} = j \right) ,& \mbox{ if $k =  x-1$, }
		\\
		\mathbb{P}\left( \mathscr{B}_{\tau_{i,-1}^{\mathscr{R}}} = j \right) ,& \mbox{ if $k \geq x$ }
	\end{cases}
\end{equation}
where the urn processes $\mathscr{B}, \mathscr{R}$ are located at $y = 1, y=0$ and $y = -1$ respectively. We note that \eqref{eq: transition prob on positive} and \eqref{eq: transition prob on negative} agree for $k \ge x$.

\subsection{Filtrations of BLPs, and Approximation of Accumulated Drifts}\label{subsec: measurability}

$\Delta^{(x,m)}_{y}$ and $\rho^{(x,m)}_{y}$ are naturally adapted to two types of filtrations. The first type of filtrations are the filtrations associated to the BLPs $\tilde{\zeta}$ and $\zeta$: 
$$\mathcal{G}_{y, +}^{(x,m)}:=\sigma\left(\mathcal{E}^{(x,m)}_{z, +}: x \le z \le y\right) $$ for $y \ge x$ and $$\mathcal{G}_{y, -}^{(x,m)}:=\sigma\left(\mathcal{E}^{(x,m)}_{z, -}: y \le z \le x\right) $$ for $y \le x$.
The second type of filtrations contain additional information of all arrival times $\tau^\mathscr{B}_{k,z}$, for all $k\leq \mathcal{E}^{(x,m)}_{z, -}$ and $x\leq z \leq y$.
For $y \ge x$, we define
\[
\mathcal{H}_{y, +}^{(x,m)} = \sigma\left( \mathcal{E}_{z, -}^{(x,m)}, \tau_{k, z}^\mathscr{B}\cdot \mathbb{1}_{\left\{ k\leq \mathcal{E}_{z, -}^{(x,m)} \right\}} : x \leq  z \leq y,  k \geq 0 \right) 
.\]
$\mathcal{H}_{y, -}^{(x,m)}$ is defined similarly for $y\leq x$.
In particular, $\mathcal{H}_{y, +}^{(x,m)}$ is finer than $\mathcal{G}_{y, +}^{(x,m)}$, since $\mathcal{E}_{z, +}^{(x,m)}$ is $\mathcal{H}_{y, +}^{(x,m)}$- measurable for any $z$ in $[x,y]$, by \eqref{eq: recursive formula for upcrossings} and \eqref{eq: source of inhomogeneity}. Similarly, $\mathcal{E}_{y, -}^{(x,m)}$ is $\mathcal{H}_{y, -}^{(x,m)}$- measurable for $ y\leq x$, so $\mathcal{H}_{y, -}^{(x,m)}$ is finer than $\mathcal{G}_{y, -}^{(x,m)}$. 

The following lemma exhibits two identities involving $\Delta^{(x,m)}_{y}$ and $\rho_{y}^{(x,m)}$. These two identities imply that $\Delta^{(x,m)}_{y}$ is $\mathcal{H}_{y, +}^{(x,m)}$- measurable but not $\mathcal{G}_{y, +}^{(x,m)}$- measurable. 
Identity \eqref{eq: cummulated drift at a site} will also be useful when we estimate $\Delta_y^{(x,m)}$ in proofs of Lemmas \ref{lm:lipchitz-bound-on-good-event} and \ref{lm: control of martingale}.
\begin{lemma}\label{lm: identities for Del, rho} 
	For any integers $m,y\geq x$ with $m\geq 0$, the conditional expectation $\Delta^{(x,m)}_{y}$ depends only on $ \mathcal{E}_{y-1,+}^{(x,m)} = \mathcal{E}_{y,-}^{(x,m)} + \mathbb{1}_{ \left\{ 1\leq y \leq x \right\} }$ and $ \tau^{\mathscr{B}}_{l,y}$ for $l\leq \mathcal{E}^{(x,m)}_{y,-} $,
	\begin{equation} \label{eq: cummulated drift at a site}
		\Delta_{y}^{(x,m)} = \sum_{l=0 }^{ \mathcal{E}^{(x,m)}_{y,-} -1  } g\left(\tau^{\mathscr{B}}_{l,y},\tau^{\mathscr{B}}_{l+1,y},l,y \right),
	\end{equation}	
	where $g(A,B,l,y)$ depends only on $A,B,l$ and the sign of $y$:
	\begin{equation}\label{eq: cummulated drift at a site2}
		g\left(\tau^{\mathscr{B}}_{l,y},\tau^{\mathscr{B}}_{l+1,y},l,y\right)
		= \sum_{j=\tau^{\mathscr{B}}_{l,y}}^{\tau^{\mathscr{B}}_{l+1,y}-1} \frac{ r_y( j-l) - b_y(l)  }{ r_y( j-l) + b_y(l)  }.
	\end{equation}
	Moreover, on the event that $\mathcal{E}^{(x,m)}_{y-1,+}  = k + \mathbb{1}_{\left\{1\leq y\leq x\right\}}$,
	\begin{equation} \label{eq: conditional mean in GPU represenetation}
		\rho_{y}^{(x,m)} = \mathbb{E}\left[ \Delta_{y}^{(x,m)}\vert \mathcal{E}^{(x,m)}_{y-1,+} \right]	  
		= \mathbb{E}\left[  \mathscr{D}_{\tau^{\mathscr{B}}_{k,y}} \right].
	\end{equation}
	Recall \eqref{eq:signed difference} for the definition of $\mathscr{D}^{(y)}_k$.	
\end{lemma}
\begin{proof} 
	The first identity is similar to \eqref{eq: recursive formula for upcrossings}.
	For any $y>x $, at time $\lambda_{x,m}$, the last jump from site $y$ is a left jump. For any integers $k\geq 0$, the event 
	$\left\{ \mathcal{E}^{(x,m)}_{y-1,+} =k +  \mathbb{1}_{\left\{1\leq y\leq x\right\}}\right\}$ equals $\left\{  L(y,\lambda_{x,m}) = \tau_{k,y}^{\mathscr{B}} \right\}, $ on which
	\[
	\Delta_{y}^{(x,m)} =\sum_{j=0}^{ L(y,\lambda_{x,m})-1} \mathbb{E}\left[ \mathscr{D}_{j+1} -\mathscr{D}_{j}  \vert \mathcal{F}^{\mathscr{B},\mathscr{R}}_{j} \right] = \sum_{j=0}^{\tau^\mathscr{B}_{k,y}-1} \mathbb{E}\left[ \mathscr{D}_{j+1} -\mathscr{D}_{j}  \vert \mathcal{F}^{\mathscr{B},\mathscr{R}}_{j} \right].  
	\] 
	Summing over terms between consecutive stopping times $\tau^{{\mathscr{B}}}_{l,y} $ and $\tau^{\mathscr{B}}_{l+1,y} $,  we get a sum depending only on $\tau^{\mathscr{B}}_{l,y} $, $\tau^{\mathscr{B}}_{l+1,y} $, $l$ and $y$:
	\begin{align} \label{eq: conditional increment}
		\sum_{j=\tau^{\mathscr{B}}_{l,y}}^{\tau^{\mathscr{B}}_{l+1,y}-1} \mathbb{E}\left[ \mathscr{D}_{j+1} - \mathscr{D}_{j}  \vert \mathcal{F}^{\mathscr{B},\mathscr{R}}_{j} \right] =&
		\sum_{j=\tau^{\mathscr{B}}_{l,y}}^{\tau^{\mathscr{B}}_{l+1,y}-1} \frac{ r_y( j-l) - b_y(l)  }{ r_y( j-l) + b_y(l)  } 
		=  g\left(\tau^{\mathscr{B}}_{l,y},\tau^{\mathscr{B}}_{l+1,y},l,y\right),
	\end{align}   
	where $(b_y(i),r_y(i))_{i\geq 0}$ is from \eqref{eq: generalized weights} and only depends on $sgn(y)$. Hence \eqref{eq: cummulated drift at a site} follows.
	
	To get \eqref{eq: conditional mean in GPU represenetation}, we note that by Remark \ref{rk:UrnGeo}, $\abs{  g\left(\tau^{\mathscr{B}}_{l,y},\tau^{\mathscr{B}}_{l+1,y},l,y\right)} \leq  \tau^{\mathscr{B}}_{l+1,y}-\tau^{\mathscr{B}}_{l,y}$ is stochastically dominated by a geometric random variable with its mean independent of $l, y$. In view of \eqref{eq: markov 1} and \eqref{eq: source of inhomogeneity}, $\left(\tau^{\mathscr{B}}_{l,y}\right)_{l \geq 0} $ is independent of $ \mathcal{E}^{(x,m)}_{y-1,+}$, and we get that 
	\begin{align*} 
		\rho_{y}^{(x,m)} 
		=& \mathbb{E}\left[ \sum_{l=0 }^{ k -1  }  g\left(\tau^{\mathscr{B}}_{l,y}\,\,,\,\tau^{\mathscr{B}}_{l+1,y} , l \right) \middle| \mathcal{E}^{(x,m)}_{y-1,+} = k + \mathbb{1}_{\left\{1\leq y\leq x\right\}} \right]	
		\\
		=&
		 \mathbb{E}\left[ \sum_{l=0}^{k-1}  \mathscr{D}_{\tau^{\mathscr{B}}_{l+1,y}} -\mathscr{D}_{\tau^{\mathscr{B}}_{l,y}} \right]  
		= \mathbb{E}\left[  \mathscr{D}_{\tau^{\mathscr{B}}_{k,y}} \right]
		.
	\end{align*}
\end{proof}

The construction of BLPs enables us to study the local time profiles at generic stopping times $\lambda_{x,m}$ from the Markov processes $\zeta$ and $\bar{\zeta}$. An advantage is that a good event $G_{n, K, t}$ (see equations \eqref{eq:good-event-1} -- \eqref{eq:good-event-4}) defined on a collection of spatial points is built from intersections of typical events $G_{n, K}^{(x,m)}(y)$ associated to single sites, see \eqref{eq:goodgood}. Then \eqref{eq: conditional mean in GPU represenetation} enables us to bound quantities in Lemmas~\ref{lm: approximation of means of local drift} and \ref{lm: approx local drift by conditional means} by looking at generalized P\'{o}lya urn process associated to a single site. In the following subsection, we recall some properties of generalized P\'{o}lya urn processes and BLPs.

\subsection{Preliminary Results}
To facilitate our arguments in section \ref{sec: approximations}, we list some results from \cite{KMP23,T96}. 
The first result is a concentration inequality for the signed difference process $\mathscr{D}_{i, y}^{(x,m)}$ in a generalized P\'{o}lya urn. This lemma is a major tool in estimating the probabilities of good events.
\begin{lemma}(Lemma 4.1 \cite{KMP23})\label{lm: concentration inequality}
	Let $y > 0$, so that the weights $r(i) = w(2i)$, $b(i)= w(2i+1) $ for all $i\geq 0$. Then there exists constants $C,c>0$ such that for $k, m \in \mathbb{N}$,
	\[
	\mathbb{P}\left(  \abs{ \mathscr{D}_{\tau_k^{\mathscr{B}}}   } \geq m \right) \leq C e^{\frac{-cm^2}{m \vee k}}.
	\]
\end{lemma} 
Lemma \ref{lm: concentration inequality} remains valid for the generalized P\'{o}lya urn process $(\mathscr{B}_{k},\mathscr{R}_{k})_{k \ge 0}$ associated to sites $y<0$ and $y=0$. For these two cases, the sequences of weights $(r(i),b(i))_{i\ge 0}$ are given by \eqref{eq: generalized weights}. The proof of this lemma uses ideas stated in Remark~\ref{rk:UrnGeo}.

The second result is an identity for a generalized P\'{o}lya urn process $(\mathscr{B}_{k},\mathscr{R}_{k})_{k \ge 0}$. For any integer $k\geq 0$ denote by $\mu(k)= \tau^{\mathscr{B}}_k - k$ the number of red balls extracted before the $k$-th blue ball. 
\begin{lemma}(Lemma 1, \cite{T96}) \label{lm: Toth's Identity}
	For any $m\in \mathbb{N}$ and $\lambda < \min\left\{ b(j): 0\leq j\leq m-1 \right\}$, we have the following identity,
	$$  \mathbb{E}\left[  \prod_{j=0}^{ \mu(m)-1 } \left(1+ \frac{\lambda}{r(j)}   \right) \right] =   \prod_{j=0}^{ m-1 } \left(1- \frac{\lambda}{b(j)}   \right)^{-1}.   $$ 
	In particular, 
	\begin{equation}\label{eq: Toth's Identity 1}
		\mathbb{E}\left[  \sum_{j=0}^{ \mu(m)-1 } \frac{1}{r(j)}   \right] =   \sum_{j=0}^{ m-1 } \frac{1}{b(j)}.
	\end{equation}	
\end{lemma}
The second identity \eqref{eq: Toth's Identity 1} is a direct consequence of the first identity, and the first identity can be proved via the following (exponential) martingale associated to the generalized P\'{o}lya urn process $(\mathscr{B}_{k},\mathscr{R}_{k})$: 
\[
\phi_k(\lambda) = \prod_{i=0}^{ \mathscr{B}_{k}-1 } \left(1-\frac{\lambda}{b(i)}\right) \prod_{j=0}^{\mathscr{R}_{k}-1 } \left(1+\frac{\lambda}{r(j)}\right)
.\]

The third result gives diffusion approximations of the branching-like processes, and it is a consequence of Proposition A.3 in \cite{KMP23}. Its proof follows arguments from T\'{o}th \cite{T96}. This result is key to the proof of process level tightness of extrema, Proposition 2.1 \cite{KMP23}. 
In the lemma below, let $\sigma_0^Z$ be the first hitting time of level $0$ for a process $\left( Z_t \right) _{t \in [0,\infty )}$ starting from a positive value, i.e. $\sigma_0^Z := \inf \{t>0: Z_t \le 0\} $.
\begin{lemma}(Proposition A.3 \cite{KMP23}; Theorem 1A \cite{T96})\label{lm: diffusion approximation of blp}
	\begin{enumerate}
		\item 
		For $n\geq 1$, let $\zeta^{(n)}=(\zeta^{(n)}_k)_{k\geq 0 }  $ be the BLP with initial value $\zeta^{(n)}_0 = \lfloor yn \rfloor$ for some $y \geq 0$, and let $\mathcal{Z}_n(t) = \frac{\zeta^{(n)}_{\lfloor nt \rfloor}}{n}$ for $n\geq 1$ and $t\geq 0$. Then we have that 
		\[
		\mathcal{Z}_n(.) \Longrightarrow Z^{(2-2\gamma)}(.)
		\] 
		as $n$ goes to infinity on $D([0,\infty))$, where $2Z^{(2-2\gamma)}(.)$ is a squared Bessel process of dimension $2-2\gamma$, with $Z^{(2 - 2 \gamma)}(0) = y$.
		
		\item
		For $n\geq 1$, let $\tilde\zeta^{(n)}=(\tilde\zeta^{(n)}_k)_{k\geq 0 }  $ be the BLP with initial value $\tilde\zeta^{(n)}_0 = \lfloor yn \rfloor$ for some $y \geq 0$, and let $\tilde{\mathcal{Z}}_n(t) = \frac{\tilde\zeta^{(n)}_{\lfloor nt \rfloor}}{n}$ for $n\geq 1$ and $t\geq 0$. Then we have that 
		\[
		\left(\tilde{\mathcal{Z}}_n(.), \sigma_0^{\tilde{\mathcal{Z}}_n}\right) 
		\Longrightarrow \left(Z^{(2\gamma)}(. \wedge \sigma_0^{Z^{(2 \gamma)}}), \sigma_0^{Z^{(2 \gamma)}}\right)
		\]
		as $n$ goes to infinity on $D([0,\infty)) \times [0,\infty )$, where $2Z^{(2\gamma)}(.)$ is a squared Bessel process of dimension $2\gamma$, with $Z^{( 2 \gamma)}(0) = y$.
	\end{enumerate}

\end{lemma}

The next two results give control of site local times from below and from above.

\begin{lemma}(Lemma 2.2 \cite{KMP23})\label{lm: number of rarely visit sites}
	Let $\gamma_+ = \gamma \vee 0$. Then for any $M>0$ and any $b>\frac{\gamma_+}{2}$, we have
	\[
	\lim_{n\to\infty} \mathbb{P}\left(\sup_{k\leq\lfloor nt \rfloor}  \sum_{x\in [I_{k-1}, S_{k-1}]} \mathbb{1}_{\left\{ L(x,k-1) \leq M \right\}} \geq 4n^b \right) = 0.
	\]
	
\end{lemma}	
The proof of Lemma~\ref{lm: number of rarely visit sites} involves analysis of BLPs and the concentration inequality for the generalized P\'{o}lya urn process in Lemma~\ref{lm: concentration inequality}. The statement of Lemma \ref{lm: number of rarely visit sites} remains in force if we replace the range $[I_{k-1}, S_{k-1}]$ by $[X_{k - 1},S_{k - 1}]$ (or $[I_{k-1},X_{k - 1}]$ respectively), and replace the local times $L(x,k-1)$ by the numbers of up-crossings $\mathcal{E}^{k}_{x,+}$ (or $\mathcal{E}^{k}_{x,-}$ respectively). After such replacements, these two statements are partial steps in the proof of Lemma 2.2 in \cite{KMP23}.   

\begin{lemma}
	\label{lm: uniform control of local time}
	\[
	\lim_{K \to  \infty } \limsup_{n \to \infty } \mathbb{P}\left( \sup_{y \in \mathbb{Z}} L\left( y, n \right) > K \sqrt{n}  \right) = 0
	.\] 
\end{lemma}
The proof uses diffusion approximations of the BLP, Lemma~\ref{lm: diffusion approximation of blp}, and mostly coincides with (Lemma~3.4, \cite{KP16}). 

\begin{proof}
	Consider the event that $L(y, n) > K \sqrt{n} $ for some $y \in \mathbb{Z}, n>0$, and $K > 0$. Define $m_{n, K } = \tau^{\mathscr{R}}_{\sqrt{K n},0 }$. Then exactly one of the following occurs:
	\begin{itemize}
		\item $L(y, \lambda_{0, m_{n, K}}) \le L(y, n)$, in which case $\lambda_{0, m_{n, K}} \le  n$;
		\item $L(y, \lambda_{0, m_{n, K}}) >  L(y, n)  > K \sqrt{n} $.
	\end{itemize}
	By symmetry we only need to control $\mathbb{P}\left( \sup _{y \ge  0} L(y, n) > K \sqrt{n}  \right) $. In view of the above two cases, we have
	\begin{align*}
		&\mathbb{P}\left( \sup _{y \ge  0} L(y, n) > K \sqrt{n}  \right) \\
		&\le \mathbb{P}\left( \lambda_{0, m_{n, K}} \le n \right) + \mathbb{P}\left( \sup _{y \ge 0} L\left( y, \lambda_{0, m_{n, K}}\right)   > K \sqrt{n}  \right)  \\
		&\le 2\mathbb{P}\left( 2 \sum_{y \ge 0} \mathcal{E}_{y, +}^{(0,m_{n, K})} \le n \right) + \mathbb{P}\left( \sup _{y \ge 0} \mathcal{E}_{y,+}^{\left( 0,m_{n, K} \right) }   > \frac{K \sqrt{n} }{2}  \right)+ \mathbb{P}\left( L(0,\lambda_{0, m_{n, K}}) > K \sqrt{n}  \right)  \\
		&\le 
		2\mathbb{P}\left( 2 \sum_{i \ge  0} \tilde \zeta_i \le  n \middle| \tilde\zeta_0 = \sqrt{K n} \right) + 
		\mathbb{P}\left( \sup _{i \ge 0} \tilde \zeta_k > \frac{K \sqrt{n} }{2} \middle| \tilde \zeta_0 = \sqrt{K n}   \right)+ 
		\mathbb{P}\left( \mathscr{R}^{(0)} _{K \sqrt{n} } < \sqrt{K n}  \right) 
		.\end{align*}
	The last probability is controlled by the concentration inequality for urn process, Lemma~\ref{lm: concentration inequality}. Using diffusion approximations of BLP, Lemma~\ref{lm: diffusion approximation of blp}, and scaling properties of BESQ processes, the first two probabilities have limits as $n$ goes to infinity :
	\begin{align*}
		\mathbb{P}\left( 2 \sum_{i \ge  0} \tilde \zeta_i \le  n \middle| \tilde\zeta_0 = \sqrt{K n} \right) 
		&\to 
		\mathbb{P}\left(2 \int _0^{\sigma_0^{Z}} Z^{(2 \gamma)}(s) d s \le \frac{1}{K} \middle| Z^{(2 \gamma)}(0) = 1 \right),\\
		\mathbb{P}\left( \sup _{i \ge 0} \tilde \zeta_k > \frac{K \sqrt{n} }{2} \middle| \tilde \zeta_0 = \sqrt{K n}   \right)
		&\to 
		\mathbb{P}\left( \sup_{s \ge 0} Z^{(2 \gamma)}(s \wedge \sigma_0^{Z}) > \frac{\sqrt{K} }{2} \middle| Z^{(2 \gamma)} (0) = 1 \right) 
		.\end{align*} 
	Both probabilities on the right hand side go to zero as $K$ goes to infinity. 
\end{proof}

\section{Approximations of Accumulated Drift}\label{sec: approximations}

In this section, we prove the technical Lemmas~\ref{lm: control of martingale}, \ref{lm: approximation of means of local drift} and \ref{lm: approx local drift by conditional means} to complete  the proof of Theorem \ref{th: main}. 
The proof of Lemma~\ref{lm: approximation of means of local drift} uses typical events that provide control on the process extrema (Proposition~\ref{prop: tightness}), regularity of urn process (Lemma~\ref{lm: concentration inequality}), and control of rarely visited sites (Lemma~\ref{lm: number of rarely visit sites}). 
In contrast, Lemma~\ref{lm: approx local drift by conditional means} requires uniform control of local time processes (Lemma~\ref{lm: uniform control of local time}), and the proof requires the diffusion approximations of BLPs (Lemma~\ref{lm: diffusion approximation of blp}).
To simplify our argument, we define the typical event $G_{n, K, t}$ in the next subsection, and apply its typicality in subsequent proofs.

\subsection{Typical Events}

For integers $K>0$, $n > 0$, and real number $t>0$, define the event
\begin{align}
	G_{n,K,t} :=  \qquad
	\label{eq:good-event-1}
	& \left\{\sup _{k \le \left\lfloor nt  \right\rfloor} |X_k| < K \sqrt{n} \right\} \cap \\
	\label{eq:good-event-2}
	& \left\{\sup_{y \in \mathbb{Z}} L(y, \left\lfloor nt  \right\rfloor) < K \sqrt{n} \right\} \cap \\
	\label{eq:good-event-3}
	& \bigcap_{y = - K \sqrt{n} }^{K \sqrt{n}} 
	\bigcap_{i = 1}^{K \sqrt{n} } \left\{\left| \tau_{i,y}^{\mathscr{B}} - 2 i \right| < \sqrt{ i } \log^2 n \right\}  \cap \\
	\label{eq:good-event-4}
	& \bigcap_{y = - K \sqrt{n} }^{K \sqrt{n}} 
	\bigcap_{i = 0}^{K \sqrt{n} } \left\{\left| \tau_{i+1, y}^{\mathscr{B}} - \tau_{i,y}^{\mathscr{B}} \right| < \log^2 n \right\}  
	.\end{align}

For each site $x \in \mathbb{Z}$ and $y > x$, and each integer $m > 0$, define the event
\begin{align}
	G_{n,K}^{(x,m)}(y) :=  \qquad
	\label{eq:good-event-5}
	& \left\{L(y,\lambda_{x,m})  < K \sqrt{n} \right\} \cap \\
	\label{eq:good-event-6}
	& \left\{\left| \tau_{i,y}^{\mathscr{B}} - 2 i \right| < \sqrt{ i } \log^2 n, \mbox{for all $i\leq  \mathcal{E}_{y,-}^{(x,m)}$} \right\}  \cap \\
	\label{eq:good-event-7}
	& \left\{\left| \tau_{i+1,y}^{\mathscr{B}} - \tau_i^{\mathscr{B},y} \right| < \log^2 n,  \mbox{for all $i\leq  \mathcal{E}_{y,-}^{(x,m)}$}  \right\}  
	.\end{align}
We see that  $\left\{G_{n, K}^{(x,m)}(y)\right\}_{y \ge x}$ is $\mathcal{H}^{(x,m)}$-adapted, i.e. $G_{n, K}^{(x,m)}(y)\in \mathcal{H}_{y, +}^{(x,m)}$.
Also note that for any $ \abs{x},m < K\sqrt{n}$, 
\begin{equation}
	\label{eq:goodgood}
	G_{n, K, t}\cap \left\{ \lambda_{x,m} \leq\lfloor nt \rfloor \right\} 
	\subset   \bigcap_{x<y< K \sqrt{n} } G_{n, K}^{(x,m)}(y) 
	\subset   \bigcap_{x<y< K \sqrt{n} } G_{n, K^2}^{(x,m)}(y)
	.\end{equation} 

\begin{lemma}
	\label{lm:good-event}
	For any $t > 0$, the event $G_{n,K,t}$ is typical in the sense that
	\[
	\lim_{K \to \infty } \limsup_{n \to \infty } 
	\mathbb{P}(G^c_{n, K,t}) = 0
	.\] 
\end{lemma}
\begin{proof}
	From the process-level tightness of extrema (Proposition~\ref{prop: tightness}), we know that the probability of event \eqref{eq:good-event-1} goes to $1$ (uniformly in $n$) as $K $ goes to infinity. The probability of the second event \eqref{eq:good-event-2} is controlled by Lemma~\ref{lm: uniform control of local time}. 
	The remaining two events \eqref{eq:good-event-3} and \eqref{eq:good-event-4} encode the asymptotic behavior of P\'{o}lya's urn processes at each site $y$. To estimate the probability of event \eqref{eq:good-event-3}, we use Lemma~\ref{lm: concentration inequality}:
	\begin{align*}
		&1-\mathbb{P}\left(\bigcap_{y = -K \sqrt{n}}^{K \sqrt{n} }\bigcap_{i = 1}^{K \sqrt{n} } \left\{\left| \tau_{i,y}^{\mathscr{B}} - 2 i \right| < \sqrt{ i } \log^2 n \right\}
		\right)
		\le \sum_{y = - K \sqrt{n} }^{K \sqrt{n}}\sum_{i = 1}^{K \sqrt{n}} \mathbb{P}\left( |\tau_{i,y}^{\mathscr{B}} - 2i| \ge \sqrt{i} \log^2 n \right) \\
		&\hspace*{16em} \le CK \sqrt{n} \sum_{i = 1}^{K \sqrt{n}} \exp\left( - c \frac{i \log^4 n}{\sqrt{i}  \log^2 n \vee i} \right)  \\
		&\hspace*{16em} \le CK \sqrt{n}  \sum_{i = 1}^{K \sqrt{n}}
		\left( \exp\left( - c \sqrt{i}  \log^2 n \right)  +
		\exp\left( - c \log^4 n \right) \right),
	\end{align*}
	and the last line goes to zero as $n$ goes to infinity, for any fixed $K>0$. 

	To control event \eqref{eq:good-event-4}, we recall Remark~\ref{rk:UrnGeo} and have
	\[
	\mathbb{P}\left(\bigcup_{y = -K \sqrt{n}}^{K \sqrt{n} }\bigcup_{i = 1}^{K \sqrt{n}}\left\{\left| \tau_{i+1,y}^{\mathscr{B}} - \tau_{i,y}^{\mathscr{B}} \right| \ge  \log^2 n \right\}\right) 
	\le C K^2 n \exp\left( - c \log^2 n \right) 
	,\] 
	which also goes to zero as $n$ goes to infinity, for any fixed $K$. We thus conclude that the probability of $G^c_{n, K, t}$ goes to zero as we first take $n$ to infinity and then take $K$ to infinity.
\end{proof}

In the next lemma, on $G_{n, K, t}$, we obtain bounds of local drifts. This allows us to apply martingale concentration inequalities to estimate accumulated drift, in the proof of Lemma~\ref{lm: approx local drift by conditional means}.
\begin{lemma}\label{lm:lipchitz-bound-on-good-event}
	For all  $y \ge x$, on $G_{n, K}^{(x,m)}(y)$, when $p \in (0,\frac{1}{2}]$,  we have
	\begin{align*}
		\left| \Delta_y^{(x,m)} \right| &\le C_K n^{-\frac{1}{2}p + \frac{1}{4}} \log^4 n &&\text{when }p \in \left(0,\frac{1}{2}\right)\\
		\left| \Delta_y^{(x,m)} \right| &\le C_K \log^5 n &&\text{when }p = \frac{1}{2}
		.\end{align*}
	As a corollary, these bounds also apply to $\left| \Delta_y^{(x,m)} - \rho_y^{(x,m)} \right| $, up to a fixed multiplicative constant.
\end{lemma}
\begin{proof}
	We present the proof of this lemma for the case $x > 0$. 
	Calculations in the cases $x = 0$ and $x < 0$ are slightly different but yield the same bound.
	For simplicity of notation, we drop $y$ from the P\'{o}lya's urn processes. We first use Lemma~\ref{lm: identities for Del, rho} to write 
	\begin{align*}
		\left| \Delta_y^{(x,m)} \right| 
		&= 
		\left| 	\sum_{i = 0}^{\mathcal{E}_{y,-}^{(x,m)}} 
		\sum_{l = \tau_i^{\mathscr{B}}} ^{\tau_{i+1}^{\mathscr{B}}  -1}
		\frac{w(2 \mathscr{R}_l)- w(2 \mathscr{B}_l + 1)}{w(2 \mathscr{R}_l)+ w(2 \mathscr{B}_l + 1)}
		\right| .
	\end{align*}

	From the identity $\frac{a - b}{a + b} = (\frac{1}{b} - \frac{1}{a})\cdot(\frac{1}{a} + \frac{1}{b})^{-1}$, and the fact that $w(.)$ is bounded, this is controlled by
	\begin{align*}
		\left| \Delta_y^{(x,m)} \right|&\le C_{p, B, \kappa} \sum_{i = 0}^{K \sqrt{n} } \sum_{l = \tau_i^{\mathscr{B}}} ^{\tau_{i+1}^{\mathscr{B}}  -1}
		\left| \frac{1}{w(2 \mathscr{R}_l)} - \frac{1}{w(2 \mathscr{B}_l + 1)} \right|  \\
		&= C_{p, B, \kappa} \sum_{i = 0}^{K \sqrt{n} } \sum_{l = \tau_i^{\mathscr{B}}} ^{\tau_{i+1}^{\mathscr{B}}  -1}
		\left| \frac{1}{w(2 l - 2 i)} - \frac{1}{w(2i + 1)} \right|  \\
		&= C_{p, B, \kappa} \sum_{i = 0}^{K \sqrt{n} }
		\sum_{j = \tau_i^{\mathscr{B}} - i}^{\tau_{i+1}^{\mathscr{B}} - i - 1} \left| \frac{1}{w(2j)} - \frac{1}{w(2i + 1)} \right|  \\
		&= C_{p, B, \kappa} \sum_{i = 0}^{K \sqrt{n} }
		\sum_{j = \tau_i^{\mathscr{B}} - i}^{\tau_{i+1}^{\mathscr{B}} - i - 1} \left|  2^p B\left( \frac{1}{(2j)^p} - \frac{1}{(2i + 1)^p} \right)  + O\left( \frac{1}{i^{\kappa + 1}} \right) + O\left( \frac{1}{j^{\kappa + 1}} \right)\right|  .\\
		\intertext{
			In view of \eqref{eq:good-event-6} and \eqref{eq:good-event-7}, we bound $\tau_{i}^{\mathscr{B}}-i$ below by $i - \sqrt{i} \log^2 n$ and bound $\tau_{i+1}^{\mathscr{B}} - i - 1$ above by $(i + 1) + \sqrt{i + 1} \log^2 n$. For $i \geq 16 \log^4 n$ we have 
			$[i - \sqrt{i} \log^2 n, i + 1 + \sqrt{i + 1} \log^2 n] \subset [i - 2 \sqrt{i} \log^2 n, i + 2 \sqrt{i} \log^2 n]$. Therefore,
		}
		&\le C_{p, B, \kappa} \left(\log^4 n +\sum_{i = 16 \log^4 n}^{K \sqrt{n} } \log^2 n \cdot \sup_{
			|j - i| \le 2 \sqrt{i}  \log^2 n
		} \left|  \frac{1}{(2j)^p} - \frac{1}{(2i + 1)^p} + O\left( \frac{1}{i^{\kappa + 1}} \right) \right|\right)\\
		&\le C_{p, B, \kappa} \log^2 n \cdot \left( \log^2 n + \sum_{i = 16 \log^4 n}^{K \sqrt{n} } \left( 
		\frac{4 \sqrt{ i } \log^2 n }{\left(2 i - 4 \sqrt{ i } \log^2 n\right)^{p + 1}} 
		+ O\left( \frac{1}{i^{\kappa + 1}} \right)
		\right) \right)  \\
		&\le C_{p, B, \kappa}' \left(\log^4 n + \sum_{i = 16 \log^4 n}^{K \sqrt{n} } i^{- p - \frac{1}{2}} \log^4 n + i^{- \kappa - 1 } \log^2 n \right)
		.\end{align*} 
	For $\kappa>0$, this gives us the desired bound.
\end{proof}

\subsection{Control of Martingale Terms, Lemma~2.1} 
\begin{proof}[Proof of Lemma~\ref{lm: control of martingale}]
	As remarked earlier in section 2, to control quadratic variation, we only need to show that
	\[
	\lim_{N \to \infty } \frac{1}{N} \sum_{k = 0}^{N-1} \mathbb{E}\left[ X_{k+1} - X_k | \mathcal{F}_k \right]^2 = 0
	\quad \text{in probability}
	.\] 
	We first decompose the sum in time to a spatial sum of local (squared) accumulated drifts:
	\begin{align}
		&\sum_{k = 0}^{N-1} \mathbb{E}\left[ X_{k+1} - X_k | \mathcal{F}_k \right]^2
		\notag
		\\
		&= \sum_{x \in \left[ I_N, S_N \right]} \sum_{j = 0}^{L(x, N) - 1} \mathbb{E}\left[ \mathscr{D}_{j+1}^{(x)} - \mathscr{D}_j^{(x)} | \mathcal{F}_{j}^{\mathscr{B}, \mathscr{R}} \right]^2  
		\label{eq:lem-martingale-1}
		.\end{align}
	In the calculation below we let $C_{\alpha_1, \alpha_2, \dots}$ denote a constant depending on  $\alpha_1, \alpha_2, \dots$ but not on $n$, and $C_{\alpha_1, \alpha_2, \dots}$ may change from line to line.
	The inner sum in \eqref{eq:lem-martingale-1} is bounded by
	\begin{align*}
		\sum_{j =0}^{ L(x, N) - 1} \mathbb{E}\left[ \mathscr{D}_{j+1}^{(x)} - \mathscr{D}_j^{(x)} | \mathcal{F}_{j}^{\mathscr{B}, \mathscr{R}} \right]^2 \stackrel{(x > 0)}{\le} C_{p, B, \kappa} \sum_{i = 0}^{\mathscr{B}_N} \sum_{l = \tau_i^{\mathscr{B}}}^{\tau_{i+1}^{\mathscr{B}}-1}
		\left| \frac{1}{w(2 \mathscr{R}_l)} - \frac{1}{w\left( 2 \mathscr{B}_l + 1 \right) } \right|^2 
		.
	\end{align*}
	To control this sum, we consider two cases depending on whether $i \ge 16 \log^4 n$ or not. For the case $i \ge 16 \log^4 n$, we have on the event $G_{n, K, t}\cap \{N \le \lfloor nt \rfloor\}$, that $|(2 l - 2 i) - (2 i + 1)| \leq 4 \sqrt{i} \log^2 n + 1$ for any $l \in [\tau_i^\mathscr{B}, \tau_{i+1}^\mathscr{B}]$, thus
	\begin{align*}
		&\sum_{i = 16 \log^4 n}^{\mathscr{B}_N} \sum_{l = \tau_i^{\mathscr{B}}}^{\tau_{i+1}^{\mathscr{B}}-1} 
		\left| \frac{1}{w(2 \mathscr{R}_l)} - \frac{1}{w\left( 2 \mathscr{B}_l + 1 \right) } \right|^2 \\
		&\leq \sum_{i = 16 \log^4 n}^{\mathscr{B}_N} \log^2 n \sup_{|l - 2 i| \le 2 \sqrt{i} \log^2 n} 
		\left| \frac{1}{w(2 l - 2 i)} - \frac{1}{w\left(2 i + 1 \right) } \right|^2 \\
		&\leq \sum_{i = 16 \log^4 n}^{\mathscr{B}_N} \log^2 n \sup_{|l - 2 i| \le 2 \sqrt{i} \log^2 n} 
		C_{p, B, \kappa} \left| (2 l - 2 i)^{-p} - (2 i + 1)^{-p} + \frac{1}{i^{1 + \kappa}} \right|^2 \\
		&\leq C_{p, B, \kappa} \sum_{i = 16 \log^4 n}^{\mathscr{B}_N} \log^2 n  
		 \left| (4 \sqrt{i} \log^2 n + 1)(2 i - 4 \sqrt{i} \log^2 n)^{- p - 1} + \frac{1}{i^{1 + \kappa}} \right|^2 \\
		&\leq C_{p, B, \kappa} \sum_{i = 16 \log^4 n}^{\mathscr{B}_N} \log^2 n \; i^{- 2 p - 1} \log^4 n \\
		&\leq C_{p, B, \kappa} \log^6 n 
			\left[(16 \log^4 n)^{- 2 p} - (\mathscr{B}_N)^{-2p}\right] \\
		&\leq C_{p, B, \kappa} (\log n)^{6 - 8 p}
		.
	\end{align*}
	When $i < 16 \log^4 n$, on $G_{n, K, t}$, we have
	\begin{align*}
		&\sum_{i = 0}^{16 \log^4 n} \sum_{l = \tau_i^{\mathscr{B}}}^{\tau_{i+1}^{\mathscr{B}}-1} 
		\left| \frac{1}{w(2 \mathscr{R}_l)} - \frac{1}{w\left( 2 \mathscr{B}_l + 1 \right) } \right|^2 
		\leq C_{p, B, \kappa} \log^4 n \; \log^2 n
		.
	\end{align*}
	In the cases $(x=0)$ and $(x < 0)$ we have exactly the same bound with minor differences in calculation.

	Thus, we can bound \eqref{eq:lem-martingale-1} on $G_{n, K, t}$ by
	\begin{equation*}
		\sum_{x \in \left[ I_N, S_N \right]} \sum_{i = 0}^{L(x,N) - 1} \mathbb{E}\left[ \mathscr{D}_{i+1}^{(x)} - \mathscr{D}_i^{(x)} | \mathcal{F}_{i}^{\mathscr{B}, \mathscr{R}} \right]^2 
		< \sum_{x \in \left[ I_N, S_N \right]} C_{p, B, \kappa} \log^6 n.
	\end{equation*}
	Finally we can take $n = N, t = 1$ and control the probability using the typical event $G_{N, K, 1}$ :
	\begin{align*}
		&\mathbb{P}\left( \left| \frac{1}{N} \sum_{k = 0}^{N-1} \mathbb{E}\left[ X_{k+1} - X_k | \mathcal{F}_k \right]^2  \right|  > \varepsilon \right)\\
		&\le \mathbb{P}\left( \left| \sum_{x \in \left[ I_N, S_N \right]} C_{p, B, \kappa} \log^6 N  \right| > \varepsilon  N \right) + \mathbb{P}\left( G_{N, K, 1}^c \right)  \\
		&\le \mathbb{1}\left(  \left| \sum_{|x| \le K \sqrt{N} } C_{p, B, \kappa} \log^6 N \right| > \varepsilon  N  \right) + \mathbb{P}\left( G_{N, K, 1}^c \right)  \\
		&\le \mathbb{1}\left(  C_{p, B, \kappa, K} \, \sqrt{N} \log^6 N > \varepsilon  N  \right) + \mathbb{P}\left( G_{N, K, 1}^c \right) 
		.\end{align*}
	For any $\epsilon > 0$ the first term vanishes trivially for large enough $N$. The second term goes to zero as $N$ goes to infinity, by Lemma~\ref{lm:good-event}. This completes the proof of the lemma.
\end{proof}

\subsection{Convergence of Conditional Expectation, Lemma~2.3}
\label{sec:RhoGamma}
Recall from Lemma~\ref{lm: identities for Del, rho}, for a site $y> x$, 
on the event $\{\mathcal{E}^{(x,m)}_{y-1,+} +\mathbb{1}_{\left\{1\leq y\leq x\right\}} = l\}$, we have
\[
\rho^{(x,m)}_y = \mathbb{E}\left[\mathscr{D}_{\tau_l^{\mathscr{B}}}^{(y)}\right]
.\]
On one hand, $S_{\left\lfloor n t\right\rfloor} \leq K\sqrt{n} $ for some $K>0$ with a high probability; on the other hand, Lemma \ref{lm: number of rarely visit sites} says that, up to $n^b$ sites (where $\frac{\gamma \vee 0}{2}<b<\frac{1}{2}$), with high probability, every site $y$ between $X_k=x$ and $\mathcal{M}^{(x,m)} =S_{k}$ satisfies $ \mathcal{E}^{(x,m)}_{y-1,+} \geq M $. To show Lemma \ref{lm: approximation of means of local drift}, it suffices to show 
\begin{equation}\label{eq: convergence of conditional expectation}
	\lim_{M\to\infty} \mathbb{E}[\mathscr{D}_{\tau_M^{\mathscr{B}}}] = \gamma , 
\end{equation} 
for positive sites. This is shown in Lemma \ref{lm: convergence of mean of discrepancies} below, and a symmetric argument after the proof of Lemma~\ref{lm: convergence of mean of discrepancies} takes care of the sites $y<0$ and $y=0$, matching the factor $\text{sgn}(y)$ in Lemma \ref{lm: approximation of means of local drift}.

\begin{lemma} \label{lm: convergence of mean of discrepancies}
	Assume $w(.)$ satisfies \eqref{eq: asymptotics of w}.
	For the generalized P\'{o}lya urn process $(\mathscr{B}_{k},\mathscr{R}_{k})$ with weights $r(i)= w(2i)$, $b(i) = w(2i+1)$ for all $i\geq 0$, we have that
	\[
	\lim_{M\to\infty} \mathbb{E}[\mathscr{D}_{\tau_M^{\mathscr{B}}}] = \gamma. 
	\]
\end{lemma} 
\begin{remark}
	Lemma \ref{lm: convergence of mean of discrepancies} is a completion of Lemma~2.4 of \cite{KMP23} dealing with the case ${p \in (\frac{1}{2}, 1]}$, where the local drift converges absolutely, allowing a simpler argument. However, when $p \in (0,\frac{1}{2}]$, the local drift is not absolutely summable in general, requiring additional care.
\end{remark}
\begin{proof} 
	Recall identities \eqref{eq: gamma} and \eqref{eq: Toth's Identity 1}. 
	Recall urn weights $r(i)$ and $b(i)$ in \eqref{eq: generalized weights}.
	For any integer $m > 0$ sufficiently large, we have
	\begin{align}
		V_1(m) - U_1(m) =& \sum_{i=0}^{m-1} \frac{1}{w(2i+1)} -\sum_{i=0}^{m-1} \frac{1}{w(2i)} 
		\notag \\
		=& \sum_{i=0}^{m-1} \frac{1}{b(i)} -\sum_{i=0}^{m-1} \frac{1}{r(i)} 
		\notag \\
		=& 	\mathbb{E}\left[  \sum_{j=0}^{ \mu(m)-1 } \frac{1}{r(j)}   \right] - \sum_{i=0}^{m-1} \frac{1}{r(i)} = \mathbb{E}\left[  \sum_{j=0}^{ \mu(m)-1 } \frac{1}{r(j)}    - \sum_{i=0}^{m-1} \frac{1}{r(i)}\right]. \label{eq: difference}
	\end{align}
	From the definition of $w(\cdot)$ in \eqref{eq: asymptotics of w}, we have that $0< \inf \frac{1}{r(j)} \leq \sup \frac{1}{r(j)} <\infty $, then $\mathbb{E}\left[\mu(m)\right]$ is bounded by
	$$\mathbb{E}\left[ \mu(m) \right] = \mathbb{E}\left[ \mu(m)\mathbb{1}_{\left\{\mu(m)\geq 1 \right\} } \right] \leq  \frac{1}{\inf 1/w(j) }\mathbb{E}\left[  \sum_{j=0}^{ \mu(m)-1 } \frac{1}{r(j)}   \right] <\infty, $$ 
	hence $ \mathbb{E}\left[ \mu(m) -m\right]  $ is finite.
	The difference in \eqref{eq: difference} can be written as
	\begin{align} 
		\sum_{j=0}^{ \mu(m)-1 } \frac{1}{r(j)} - \sum_{i=0}^{m-1} \frac{1}{r(i)} =& \sum_{j=m}^{\mu(m)-1} \left(\frac{1}{r(j)} -\frac{1}{r(m)} \right) \cdot\mathbb{1}_{\left\{\mu(m)\geq m\right\}} 
		\label{eq: 1st term}
		\\	
		& - \sum_{j=\mu(m)}^{m-1} \left(\frac{1}{r(j)} -\frac{1}{r(m)} \right) \cdot \mathbb{1}_{\left\{\mu(m)< m\right\}} 
		\label{eq: 2nd term}
		\\
		& + \frac{\mu(m)-m}{ r(m) }. \label{eq: major term}
	\end{align} 

	Since $\mu(m)-m = \mathscr{D}_{\tau^{\mathscr{B}}_m}$, the last term \eqref{eq: major term} is exactly $\frac{1}{r(m)} \mathscr{D}_{\tau^{\mathscr{B}}_m}$, which has an expectation $\frac{1}{r(m)} \mathbb{E}\left[\mathscr{D}_{\tau^{\mathscr{B}}_m}\right]$. 
	
	Let $A> \frac{2}{c} \vee 1$, where $c$ is from Lemma \ref{lm: concentration inequality}. \eqref{eq: 1st term} is bounded by
	\begin{align*}
		& \sum_{j=m}^{\infty} \abs{\frac{1}{r(j)} -\frac{1}{r(m)} } \cdot \mathbb{1}_{\left\{\mu(m)\geq j \right\}}\\
		& \leq  \sum_{0\leq j-m \leq A \sqrt{m}\log m } \abs{\frac{1}{r(j)} -\frac{1}{r(m)} } \cdot \mathbb{1}_{\left\{\mu(m)\geq j \right\}}
		+  \sum_{j-m > A\sqrt{m}\log m } \abs{\frac{1}{r(j)} -\frac{1}{r(m)} } \cdot \mathbb{1}_{\left\{\mu(m)\geq j \right\}}
		\notag
		\\
		&\leq  \sum_{0\leq j-m \leq A\sqrt{m}\log m } \abs{\frac{1}{r(j)} -\frac{1}{r(m)} }
		\quad +\quad 2\left(\sup_j \frac{1}{w(j)}\right) \cdot \sum_{j\geq m + A\sqrt{m}\log m } \mathbb{1}_{\left\{ \mu(m) > j \right\}}
		\label{eq: low large difference}
	\end{align*}
	By \eqref{eq: asymptotics of w}, there is a constant $C'>0$ such that for any $m$ sufficiently large and any $j$ with $\abs{j-m}\leq A \sqrt m \log m $, 
	\[ \abs{\frac{1}{r(j)} -\frac{1}{r(m)} } \leq C' A m^{-p-\frac{1}{2}} \log m, \]
	which implies
	\[
	\sum_{0\leq j-m \leq A\sqrt{m}\log m } \abs{\frac{1}{r(j)} -\frac{1}{r(m)} } \le 
	C' A^2 m^{-p} (\log m)^2.
	\] On the other hand, Lemma \ref{lm: concentration inequality} implies that
	\begin{align*}
		2\left(\sup_j \frac{1}{w(j)}\right) \mathbb{E}\left[\sum_{j\geq m + A\sqrt{m}\log m } \mathbb{1}_{\left\{ \mu(m) > j \right\}} \right]
		\le & 2\left(\sup_j \frac{1}{w(j)}\right) \sum_{j-m \geq A \sqrt m \log m  } \mathbb{P}( \mathscr{D}_{\tau^{\mathscr{B}}_m} \geq j-m )  
		\notag 
		\\
		\leq& 2\left(\sup_j \frac{1}{w(j)}\right) \sum_{l \geq A \sqrt m \log m } C \exp\left( - \frac{c  \cdot l^2}{l \vee m}   \right)
		\notag\\
		\leq& C'' \left( \exp (- cA^2 \cdot \log m ) + \exp(-cA \cdot \log m) \right), 
	\end{align*} for some $C''$ independent of $m$. Therefore, the expectation of \eqref{eq: 1st term} is bounded by
	\begin{equation}\label{eq: boound}
		C' A^2 m^{-p} \log m + C''  \left( m ^{-cA^2} +  m^{-cA} \right). 
	\end{equation}

	The final term \eqref{eq: 2nd term} can be treated similarly. Therefore,
	$$ \abs{ V_1(m)- U_1(m) -\frac{1}{r(m)}\mathbb{E}\left[ \mathscr{D}_{\tau^{\mathscr{B}}_m} \right] }
	\leq 2C' A^2 m^{-p} \log m + 2C''  \left( m ^{-cA^2} +  m^{-cA} \right), 
	$$ 
	which converges to zero as $m$ goes to infinity. We conclude that 
	\[
	\lim_{m\to\infty}\mathbb{E}\left[ \mathscr{D}_{\tau^{\mathscr{B}}_m} \right] = \gamma, 
	\] 
	from $\lim_{m\to\infty}\frac{1}{r(m)} =1$ and $ \lim_{m\to \infty} \left(V_1(m)-U_1(m) \right) = \gamma$.
\end{proof}

For the generalized P\'{o}lya urn process associated to a site $y<0$, we can use spatial symmetry of urn processes discussed in Remark~\ref{rm:symmetry}. Specifically, we have $r(i) = w(2i+1)$, $b(i) =w(2i)$. \eqref{eq: difference} becomes $U_1(m)-V_1(m)$, which converges to $-\gamma$ as $m$ goes to infinity.

\begin{equation}\label{eq: general expected drift}
	\lim_{m\to\infty}\mathbb{E}\left[ \mathscr{D}_{\tau^{\mathscr{B}}_{m,-1}} \right] = \lim_{m\to\infty}\mathbb{E}\left[ \mathscr{D}_{\tau^{\mathscr{R}}_{m,1}} \right] = -\gamma.
\end{equation}
Similarly, for $y=0$, the associated urn process has 
$\lim_{m\to\infty}\mathbb{E}\left[ \mathscr{D}_{\tau^{\mathscr{B}}_m} \right] = 0.$

%
Now we are ready to show the following more general version of Lemma \ref{lm: approximation of means of local drift}:
\begin{lemma}
	Let $w(.)$ be a positive monotone function on $\mathbb{N}_0$ satisfying \eqref{eq: asymptotics of w}. Then for any $0<p<1$, any $\varepsilon>0$,
	\[
	\lim_{n\to\infty} \mathbb{P}\left( \sup_{k\leq \left\lfloor n t\right\rfloor}  \abs{  	\sum_{y\in \left[X_{k}+1 ,S_{k}\right]} \left( \rho^{(X_k,L(X_k,k)-1)}_y -  \gamma \cdot sgn(y) \right) } \geq  \varepsilon \sqrt{n}     \right) =0.
	\]
	Furthermore,
	\[
	\lim_{n\to\infty} \mathbb{P}\left( \sup_{k\leq \left\lfloor n t\right\rfloor}  \abs{  	\sum_{y\in \left[I_k ,S_{k}\right]} \left( \rho^{(X_k,L(X_k,k)-1)}_y -  \gamma \cdot sgn(y) \right) } \geq  \varepsilon \sqrt{n}     \right) =0.
	\]
\end{lemma}
\begin{proof} There are only three types of weight sequences for the generalized P\'{o}lya urn processes, see \eqref{eq: generalized weights}. Therefore, from Lemma \ref{lm: convergence of mean of discrepancies}, and 
	\eqref{eq: general expected drift}, there is a decreasing function $C(.)$ on $\mathbb{N}_0$ with $\lim_{L\to \infty}C(L) =0$ such that for any $y \in \mathbb{Z}$,
	\begin{equation}\label{eq: uniform convergence}
		\max\left(\abs{\mathbb{E}\left[ \mathscr{D}_{\tau_L^{\mathscr{R}}} \right] + \gamma \cdot sgn(y)}, \abs{\mathbb{E}\left[ \mathscr{D}_{\tau_L^{\mathscr{B}}} \right] - \gamma \cdot sgn(y)} \right)\leq C(L).
	\end{equation} One such function is $C(l) = \sup \left\{  \abs{\mathbb{E}\left[ \mathscr{D}_{\tau_m^{\mathscr{R}}} \right] + \gamma \cdot sgn(y)} + \abs{\mathbb{E}\left[ \mathscr{D}_{\tau_m^{\mathscr{B}}} \right] - \gamma \cdot sgn(y)} : m\geq l \right\}.     $

	Let $t>0$ and $b \in \left[\frac{\gamma \vee 0 }{2},\frac{1}{2}\right)$. For any $n,K,M>0$, we consider two classes of good events corresponding to controls of extrema and rarely visited sites:
	\begin{align*}
		A_{n,K}:=&\left\{ \min\left\{-I_{\left\lfloor n t\right\rfloor}, S_{\left\lfloor n t\right\rfloor}\right\} \geq K \sqrt{n}  \right\}
		\\
		B_{n,M}:=& \left\{  \sup_{k\leq \left\lfloor n t\right\rfloor} \sum_{ y\in (X_{k-1}, S_{k-1}]}  \mathbb{1}_{\left\{ \mathcal{E}^{k-1}_{y-1,+} \leq M  \right\}} >n^b  \right\}.
	\end{align*}
	Clearly, $A_{n,K}$ is decreasing in $K$, and $B_{n,M}$ is increasing in $M$. We claim that for $n$ sufficiently large, the event 
	\[
	F_{n,\varepsilon}:= \left\{ \sup_{k\leq \left\lfloor n t\right\rfloor}  \abs{  	\sum_{y\in [X_{k}+1 ,S_k]} \left( \rho^{(X_k,L(X_k,k)-1)}_y -  \gamma \cdot sgn(y) \right) } \geq  \varepsilon \sqrt{n}    \right\}
	\] 
	is contained in $A_{n,K} \cup B_{n,M} $ for some finite $K, M$ independent of $n$:   
	Indeed, depending on whether $\mathcal{E}^{k-1}_{y-1,+} \leq M$ or not, terms  $\left( \rho^{(X_k,L(X_k,k)-1)}_y -  \gamma \cdot sgn(y) \right)$ are bounded by $C(0)$ or $C(M)$. 
	
	Therefore, we have
	\begin{align*}
		\sup_{k\leq \left\lfloor n t\right\rfloor}  \abs{  	\sum_{y\in [X_{k}+1 ,S_k]} \left( \rho^{(X_k,L(X_k,k)-1)}_y -  \gamma \cdot sgn(y) \right) } \leq &  
		C(0) \cdot \sup_{k\leq \left\lfloor n t\right\rfloor} \left\{   	\sum_{y\in [X_{k}+1 ,S_k]} \mathbb{1}_{ \left\{ \mathcal{E}^{k-1}_{y-1,+} \leq M \right\} } \right\}
		\notag
		\\
		+& C(M) \cdot \sup_{k\leq \left\lfloor n t\right\rfloor} \left\{   	\sum_{y\in [X_{k}+1 ,S_k]} \mathbb{1}_{ \left\{ \mathcal{E}^{k-1}_{y-1,+} \geq M \right\} } \right\},
	\end{align*} which is bounded on $A^c_{n,K} \cap B^c_{n,M}$ by
	\begin{equation}\label{eq: an upper bound on good set}
		C(0)n^b  + C(M) \left(K \sqrt{n} -n^b\right).
	\end{equation} 
	As $n$ goes to infinity, \eqref{eq: an upper bound on good set} is smaller than $\varepsilon \sqrt{n}$ for any $K>0$ and any $M$ such that $C(M) < \frac{\varepsilon}{2K}$. 
	For any such pair of $(K,M)$, $A^c_{n,K} \cap B^c_{n,M} \subset F^c_{n,\varepsilon}$ when $n$ is sufficiently large, and 
	\[
	\limsup_{n\to \infty} \mathbb{P}(F_{n,\varepsilon}) \leq \limsup_{n\to \infty}  \mathbb{P}(A_{n,K}) +  \limsup_{n\to \infty}  \mathbb{P}(B_{n,M}).
	\]
	In view of Lemma \ref{lm: number of rarely visit sites} and the explanation after it, the second limit is zero. The first limit $\limsup_{n\to \infty}  \mathbb{P}(A_{n,K}) $ vanishes as $K$ goes to infinity, which is a consequence of Lemma 2.1 \cite{KMP23}, or Corollary 1A \cite{T96}.
\end{proof}


\subsection{Approximation of Local Drifts by Conditional Means, Lemma 2.4}
\label{sec:DeltaRho}
In this subsection, we prove Lemma~\ref{lm: approx local drift by conditional means} and thus complete the proof of Theorem~\ref{th: main}. This proof is similar to, and slightly more technical than the proof of {Lemma~4.2} in \cite{KP16}. 

In Lemma~\ref{lm: approx local drift by conditional means}, for any fixed $(x,m)$ with $\abs{x},m < K\sqrt{n}$, the sum $\sum_{z=x+1}^{y}  \Delta_z^{(x,m)} - \rho_z^{(x,m)}  $ is a martingale (indexed by $y$). To control the sum, we compare the martingale $\sum_{z=x+1}^{y} \Delta_z^{(x,m)} - \rho_z^{(x,m)}$ to a tempered version $\sum_{z= x+1}^y \tilde \Delta_{z}^{(x,m)} - \tilde\rho_z^{(x,m)}$ that has bounded increments on the good event $G_{n, K, t} \cap \left\{\lambda_{x,m} \leq\lfloor nt \rfloor \right\}$. We define
\[
\tilde \Delta_y^{(x,m)} := \Delta_y ^{(x,m)} \mathbb{1}\left( G_{n, K^2}^{(x,m)} (y)\right), \qquad
\tilde \rho_y^{(x,m)} := \mathbb{E}\left[ \tilde\Delta_y^{(x,m)} \middle| \mathcal{H}_{y-1, +}^{(x,m)} \right]  
.\] 

At $\lambda_{x, m}$, the sum in Lemma~\ref{lm: approx local drift by conditional means} can be decomposed as follows:
\begin{align}
	\label{eq:tempered-difference-0}
	&\sum_{y > x} \Delta_y^{(x,m)} -  \rho_y^{(x,m)}  \\
	\label{eq:tempered-difference-1}
	&= \sum_{y > x} \Delta_y^{(x,m)} - \tilde\Delta_y^{(x,m)} \\
	\label{eq:tempered-difference-2}
	&+ \sum_{y > x} \tilde \Delta_{y}^{(x,m)} - \tilde\rho_y^{(x,m)} \\
	\label{eq:tempered-difference-3}
	&+ \sum_{y > x} \tilde\rho_y^{(x,m)} - \rho_y^{(x,m)} 
	.\end{align}

\begin{proof}[Proof of Lemma~\ref{lm: approx local drift by conditional means}]
	We first control \eqref{eq:tempered-difference-0} on $G_{n, K, t} \cap \left\{\lambda_{x,m} \leq\lfloor nt \rfloor \right\}$. From \eqref{eq:goodgood} and that $G_{n,K}^{(x,m)}(y)$ is increasing in $K$, we get that  \eqref{eq:tempered-difference-1} is zero on $G_{n, K, t} \cap \left\{\lambda_{x,m} \leq\lfloor nt \rfloor \right\}$. 
	
	For \eqref{eq:tempered-difference-2}, it equals to
	$
	\sum_{y > x}^{K\sqrt{n}} \tilde \Delta_{y}^{(x,m)} - \tilde\rho_y^{(x,m)}
	$
	on $G_{n, K, t} \cap \left\{\lambda_{x,m} \leq\lfloor nt \rfloor \right\}$. By Azuma's inequality and Lemma~\ref{lm:lipchitz-bound-on-good-event}, we get that for any $\varepsilon>0$,
	\begin{align}
		\mathbb{P}\left( \left| \sum_{y = x + 1}^{K \sqrt{n} } (\tilde\Delta_y^{(x,m)} - \tilde\rho_y^{(x,m)}) \right| > \varepsilon \sqrt{n}  \right) 
		&
		\label{eq:azuma-drift-martingale}
		\le \begin{cases}
			\exp\left( - C_{K, \varepsilon} \, n^{p } \log^{-8} n \right)
			& 0 < p < \frac{1}{2} \\
			\exp\left( - C_{K, \varepsilon} \, \sqrt{n}  \log^{-10} n \right)
			& p = \frac{1}{2}
		\end{cases}
		.\end{align}
	
	It remains to control  \eqref{eq:tempered-difference-3} on $G_{n, K, t} \cap \left\{\lambda_{x,m} \leq\lfloor nt \rfloor \right\}$:
	\begin{align*}
		\sum_{y > x} \tilde\rho_y^{(x,m)} - \rho_y^{(x,m)}
		&= \sum_{y = x + 1}^{K \sqrt{n} } \mathbb{E}\left[ \Delta_y^{(x,m)}\mathbb{1}\left( G_{n, K^2}^{(x,m)}(y) \right) - \Delta_{y}^{(x,m)} \middle| \mathcal{H}_{y-1, +}^{(x,m)}  \right]  \\
		&= \sum_{y = x + 1}^{K \sqrt{n} } -\mathbb{E}\left[ \Delta_y^{(x,m)}\mathbb{1}\left( \left( G_{n, K^2}^{(x,m)}(y) \right) ^c \right) \middle| \mathcal{H}_{y-1, +}^{(x,m)}  \right]
	.\end{align*}
	Hence,
	\begin{align*}
	&\mathbb{1}(G_{n, K, t}) \left| 
	\sum_{y > x} \tilde\rho_y^{(x,m)} - \rho_y^{(x,m)}
	\right|
	\le \sum_{y = x + 1}^{K \sqrt{n} } \mathbb{1}\left(G_{n, K}^{(x,m)}(y-1)\right) 
	\left|  \mathbb{E}\left[ \Delta_y^{(x,m)}\mathbb{1}\left( \left( G_{n, K^2}^{(x,m)}(y) \right) ^c \right) \middle| \mathcal{H}_{y-1, +}^{(x,m)}  \right] \right| \\
	& \le \sum_{y = x + 1}^{K \sqrt{n} } 
	\sqrt{
	\mathbb{E}\left( \mathbb{1}\left(\left( G_{n, K^2}^{(x,m)}(y) \right) ^c \right)
	\mathbb{1}\left( G_{n, K}^{(x,m)}(y-1) \right) 
	\middle| \mathcal{H}_{y-1, +}^{(x,m)} \right)
	}\\
	&\hspace{18em} \times \sqrt{
	\mathbb{E}\left[ \left(\Delta_y^{(x,m)}\right)^2
	\mathbb{1}\left(G_{n, K}^{(x,m)}(y-1)\right) 
	\middle| \mathcal{H}_{y-1, +}^{(x,m)}  \right]}
	.
	\end{align*}
	Thus, the first expectation is bounded by
	\begin{align*}
		&\mathbb{E}\left[ \mathbb{1}\left(\left( G_{n, K^2}^{(x,m)}(y) \right) ^c \right) \mathbb{1}\left( G_{n, K}^{(x,m)}(y-1) \right) \middle| \mathcal{H}_{y-1, +}^{(x,m)} \right] \\
		&\le \sum_{a = 0}^{K \sqrt{n}} \mathbb{E}\left[ \mathbb{1}\left(\left( G_{n, K^2}^{(x,m)}(y) \right) ^c \right) \mathbb{1}\left( \mathcal{E}_{y-1, +}^{(x,m)} = a \right) \middle| \mathcal{H}_{y-1, +}^{(x,m)} \right] \\
		&\le \sum_{a = 0}^{K \sqrt{n}} \mathbb{P}\left( \left( G_{n, K^2}^{(x,m)}(y) \right) ^c \middle| \mathcal{E}_{y-1, +}^{(x,m)} = a \right) \mathbb{1}\left( \mathcal{E}_{y-1, +}^{(x,m)} = a \right) \\
		&\le \exp\left( - 2 K^2 \sqrt{n} (q - \frac{1}{K}) \right) + C K^2 n \exp\left(- c \log^2 n\right)
		.
	\end{align*}
	To justify the last inequality,
	we use the union bound
	\begin{align*}
	\mathbb{P}\left( \left( G_{n, K^2}^{(x,m)}(y) \right) ^c \middle| \mathcal{E}_{y-1, +}^{(x,m)} 
	= a \right)
	&\le \mathbb{P}\left( \mathcal{E}_{y, +}^{(x,m)} \ge K^2 \sqrt{n} \middle| \mathcal{E}_{y-1, +}^{(x,m)} = a \right) \\
	&+ \mathbb{P}\left( \left\{\left|\tau_{i, y}^{\mathscr{B}}-2 i\right|\ge\sqrt{i} \log ^2 n, \text { some } i \leq K^2 \sqrt{n}\right\} \right) \\
	&+ \mathbb{P}\left( \left\{\left|\tau_{i+1, y}^{\mathscr{B}}-\tau_i^{\mathscr{B}, y}\right|\ge\log ^2 n, \text { some } i \leq K^2 \sqrt{n}\right\} \right)
	\end{align*}
	and control the second and third probabilities as in the proof of Lemma~\ref{lm:good-event} for the complements of \eqref{eq:good-event-6} and \eqref{eq:good-event-7}. For the first probability,
	we use Remark \ref{rk:UrnGeo} and recursive formulas \eqref{eq: recursive formula for upcrossings}, \eqref{eq: source of inhomogeneity}. As increments $\tau^{\mathscr{B}}_{k+1,y} -\tau^{\mathscr{B}}_{k,y}$ are stochastically dominated by independent geometric random variables with parameter $q$ uniformly in $ k \geq 0$ and $y \in \mathbb{Z}$, $\mathcal{E}_{x+1,+}^{(x,m)}$ is stochastically dominated by the sum of $ (\mathcal{E}_{x,+}^{(x,m)} +1) $ i.i.d. geometric random variables with parameter $q$. Then for all sufficiently large $K$, and all integers $a\leq K\sqrt{N}$, we have the following upper bounds for $\mathbb{P}\left( G_{n, K}^{(x,m)}(y-1)\cap \left\{L(y, \lambda_{x, m}) \ge K \sqrt{n}\right\} \middle| \mathcal{E}_{y, +}^{(x,m)}=a \right)$:
	\begin{align*}
		\mathbb{P}\left(\mathcal{E}_{x+1,+}^{(x,m)} > \frac{K^2}{2} \sqrt{n} | \mathcal{E}_{x,+}^{(x, m)} = \lceil K \sqrt{n} \rceil\right)
		&\le \exp\left( - 2 K^2 \sqrt{n}(q - \frac{1}{K})  \right) 
		,
	\end{align*}
	which converges to zero exponentially fast in $K \sqrt{n}$. 
	
	Next, to control the second expectation, we observe from Lemma~\ref{lm: identities for Del, rho} and \eqref{eq: recursive formula for upcrossings} that $\Delta_y^{(x,m)} \in \sigma \left(\mathcal{E}_{y, +}^{(x,m)}, (\tau_{y,i}^{\mathscr{B}})_{i \ge 0}\right)$, and furthermore, $\mathcal{E}_{y, +}^{(x,m)} \in \sigma \left(\mathcal{E}_{y-1, +}^{(x,m)}, (\tau_{y,i}^{\mathscr{B}})_{i \ge 0}\right)$. Hence $\mathbb{E}\left[|\Delta_y^{(x,m)}|^2 \middle| \mathcal{H}_{y-1, +}^{(x,m)}\right] = \mathbb{E}\left[|\Delta_y^{(x,m)}|^2 \middle| \mathcal{E}_{y-1, +}^{(x,m)}\right]$. Therefore,
	\begin{align*}
		& \mathbb{E}\left[ \left(\Delta_y^{(x,m)}\right)^2 \mathbb{1}\left(G_{n, K}^{(x,m)}(y-1)\right) \middle| \mathcal{H}_{y-1, +}^{(x,m)}  \right] \\
		&\le \sum_{a = 0}^{K \sqrt{n}} \mathbb{E}\left[ \left(\Delta_y^{(x,m)}\right)^2 \mathbb{1}\left( \mathcal{E}_{y-1, +}^{(x,m)} = a \right) \middle| \mathcal{H}_{y-1, +}^{(x,m)} \right] \\
		&\le \sum_{a = 0}^{K \sqrt{n}} \mathbb{E}\left[ \left(\Delta_y^{(x,m)}\right)^2 \middle| \mathcal{E}_{y-1, +}^{(x,m)} = a \right] \mathbb{1}\left( \mathcal{E}_{y-1, +}^{(x,m)} = a \right) \\
		&\le C_q K^2 n
	\end{align*}
	To justify the last inequality, note that $\abs{\Delta_{x+1}^{(x,m)}} \leq  \mathcal{E}_{x,+}^{(x,m)}+1 + \mathcal{E}_{x+1,+}^{(x,m)}$,  
	so that $ \left(\Delta_{x+1}^{(x,m)} \right)^2$ is stochastically dominated by a sum of $\left(\mathcal{E}_{x,+}^{(x,m)}+1\right)^2$ products. Hence,
	\begin{align*}
		\mathbb{E}\left[ \left(\Delta_{x+1}^{(x,m)}\right)^2 \middle| \mathcal{E}_{x,+}^{(x,m)} = a  \right]  \leq C_q a^2
	\end{align*}
	for some $C_q$ depending only on $q$.

	%
	Therefore, on $G_{n,K,t} \cap \left\{ \lambda_{x,m} \leq\lfloor nt \rfloor \right\}$,
	\begin{equation}\label{eq: difference of cond means}
		\left| \sum_{y > x} \tilde\rho_y^{(x,m)} - \rho_y^{(x,m)} \right| \le 
		C_q K^2 n \exp \left( - c \log^2 n \right) 
	\end{equation}
	which is smaller than $\frac{\varepsilon \sqrt{n}}{2}$ for $n$ sufficiently large.
	
	

Finally, to prove the lemma, we use a union bound by considering possible values $(x,m)$ in the range of $\left(X_k, L\left(X_k, \left\lfloor nt  \right\rfloor - 1\right)\right)$. Note that if $G_{n,k,t}$ occurs, for any $k\leq \lfloor nt \rfloor$, there is a pair $(x,m)$ with $\lambda_{x,m}=k$ such that $\abs{x},m <K\sqrt{n}$. Therefore, we get from \eqref{eq:azuma-drift-martingale} and \eqref{eq: difference of cond means} that 
\begin{align*}
	& \mathbb{P}\left( \sup_{k <\lfloor nt \rfloor} \left| \sum_{y > X_k} 
	\Delta_y^{\left(X_k,L(X_k, k)-1\right)} - \rho_y^{\left(X_k,L(X_k, k)-1\right)}
	\right| > \varepsilon \sqrt{n}  \right) \\
	&\le \mathbb{P}(G_{n, K, t}^c) + \mathbb{P}\left( \bigcup_{|x|, m < K \sqrt{n} } \left\{  \left| \sum_{y > x} \Delta_y^{(x,m)} - \rho_y^{(x,m)} \right|  > \varepsilon \sqrt{n},  \lambda_{x,m} \leq\lfloor nt \rfloor \right\} \cap G_{n,K,t} \right) \\
	&\le \mathbb{P}(G_{n, K, t}^c) + K^2 n \sup _{|x|, m \le  K \sqrt{n} }
	\mathbb{P}\left( \left| \sum_{y \ge x} \Delta_y^{(x,m)} - \rho_y^{(x,m)} \right|  > \varepsilon \sqrt{n} , G_{n,K,t}\cap \left\{\lambda_{x,m} \leq\lfloor nt \rfloor \right\}  \right) \\
	&\le \mathbb{P}(G_{n, K, t}^c) + K^2 n \left( \exp\left( - C_{K, \varepsilon} \, n^{p } \log^{-8} n \right) + \exp\left( - C_{K, \varepsilon} \, \sqrt{n}  \log^{-10} n \right)\right) 
	.\end{align*}
In view of Lemma~\ref{lm:good-event}, the last line vanishes as we first take $n$ to infinity and then take $K$ to infinity.
\end{proof}

\bibliographystyle{amsalpha}


\providecommand{\bysame}{\leavevmode\hbox to3em{\hrulefill}\thinspace}
\providecommand{\MR}{\relax\ifhmode\unskip\space\fi MR }
\providecommand{\MRhref}[2]{%
  \href{http://www.ams.org/mathscinet-getitem?mr=#1}{#2}
}
\providecommand{\href}[2]{#2}


\end{document}